\documentclass[11pt]{amsart}

\usepackage{amsmath,amsthm,amsfonts,amssymb,mathrsfs,mathdots}
\usepackage[margin=1in]{geometry}
\usepackage[bookmarks]{hyperref}
\usepackage{verbatim}
\usepackage{appendix}
\usepackage{todonotes}
\usepackage{tikz}
\usetikzlibrary{matrix,arrows}
\usepackage{fancyvrb }
\usepackage{color}
\usetikzlibrary{arrows}

\newtheorem*{theo}{Theorem}
\newtheorem*{coro}{Corollary}

\newtheorem{theorem}{Theorem}[section]

\newtheorem{proposition}[theorem]{Proposition}

\newtheorem*{conjecture}{Conjecture}

\newtheorem{corollary}[theorem]{Corollary}

\newtheorem{lemma}[theorem]{Lemma}
\input xy
\xyoption{all}

\theoremstyle{definition}

\newtheorem{prop}[theorem]{Proposition}

\newtheorem{remark}[theorem]{Remark}

\DeclareMathOperator{\rank}{rank}

\DeclareMathOperator{\chr}{char}

\DeclareMathOperator{\Ext}{Ext}

\DeclareMathOperator{\opH}{H}

\DeclareMathOperator{\Lie}{Lie}

\DeclareMathOperator{\red}{red}

\DeclareMathOperator{\GL}{GL}
\DeclareMathOperator{\SL}{SL}
\DeclareMathOperator{\SO}{SO}

\DeclareMathOperator{\ad}{ad}
\DeclareMathOperator{\Ad}{Ad}

\newcommand{\calO}{\mathcal{O}}

\newcommand{\calV}{\mathcal{V}}

\newcommand{\spec}{\text{Spec}}

\newcommand{\g}{\mathfrak{g}}

\newcommand{\fraku}{\mathfrak{u}}
\newcommand{\z}{\mathfrak{z}}

\newcommand{\Hbul}{\opH^\bullet}                                  

\newcommand{\N}{\mathcal{N}}

\newcommand{\0}{\mathcal O}

\begin{document}

\title[Commuting varieties and cohomological complexity theory]{Commuting varieties and cohomological complexity theory}

\author{Paul D. Levy}
\address{Department of Mathematics and Statistics \\ Lancaster University \\ Lancaster \\ LA1 4YW, UK}
\email{p.d.levy@lancaster.ac.uk}

\author{Nham V. Ngo}
\address{Department of Mathematics\\ University of North Georgia--Gainesville \\ Oakwood\\ GA~30566, USA}
\email{nvngo@ung.edu}

\author{Klemen \v{S}ivic}
\address{University of Ljubljana, Faculty of mathematics and physics, Jadranska ulica 19, 1000 Ljubljana, Slovenia}
\email{klemen.sivic@fmf.uni-lj.si}

\begin{abstract}
In this paper we determine, for all $r$ sufficiently large, the irreducible component(s) of maximal dimension of the variety of commuting $r$-tuples of nilpotent elements of $\mathfrak{gl}_n$.
Our main result is that in characteristic $\neq 2,3$, this nilpotent commuting variety has dimension $(r+1)\lfloor \frac{n^2}{4}\rfloor$ for $n\geq 4$, $r\geq 7$.
We use this to find the dimension of the (ordinary) $r$-th commuting varieties of $\mathfrak{gl}_n$ and $\mathfrak{sl}_n$ for the same range of values of $r$ and $n$.

Our principal motivation is the connection between nilpotent commuting varieties and cohomological complexity of finite group schemes, which we exploit in the last section of the paper to obtain explicit values for complexities of a large family of modules over the $r$-th Frobenius kernel $(\GL_n)_{(r)}$.
These results indicate an inequality between the complexities of a rational $G$-module $M$ when restricted to $G_{(r)}$ or to $G(\mathbb F_{p^r})$; we subsequently establish this inequality for every simple algebraic group $G$ defined over an algebraically closed field of good characteristic, significantly extending the main theorem in \cite{LN}.   
\end{abstract}

\maketitle

\section{Introduction}

\subsection{Commuting Varieties}

Let $k$ be an algebraically closed field and let $r$ be a positive integer.
Let $G$ be a linear algebraic group defined over $k$ and $\g=\Lie(G)$.
For each closed subvariety $V$ of $\g$, denote by
\[ C_r(V)=\{ (x_1,\ldots,x_r)\in V^r:[x_i,x_j]=0, 1\le i,j \le r \} \]
the variety of commuting $r$-tuples of elements of $V$.
For brevity, we will call $C_r(V)$ the $r$-th commuting variety of $V$, or just the commuting variety of $V$ if there is no potential for confusion.
Two special cases which have received considerable interest are the {\it ordinary commuting variety} $C_r(\g)$, and the {\it nilpotent commuting variety} $C_r(\N(\g))$, where $\N(\g)$ is the nilpotent cone of $\g$.
Basic geometric properties of these varieties, such as the number and dimensions of irreducible components, are unknown except for small values of $r$ or $n$.

The first results on ordinary commuting varieties were the independent proofs by Motzkin-Taussky \cite{M-T:1995} and Gerstenhaber \cite{G:1961} that  $C_2(\mathfrak{gl}_n)$ is irreducible.
Subsequently, this was extended to the Lie algebra of an arbitrary reductive algebraic group by Richardson \cite{R:1979} in characteristic zero, and by the first author \cite{L:2003} in positive characteristic (under mild hypotheses).
It was observed in \cite{G:1961} that the irreducibility of $C_r(\mathfrak{gl}_n)$ can fail for $r\geq 4$.
Guralnick showed \cite{Gu:1992} that $C_r(\mathfrak{gl}_n)$ is irreducible for $n\le 3$ and any $r$, while it is reducible for all $n, r\ge 4$.

Although nilpotent commuting varieties have received rather less attention until relatively recently, interest in $C_2(\N(\mathfrak{gl}_n))$ has been stimulated by a close relationship with the punctual Hilbert scheme of $n$ points in the plane.
This connection was exploited by Baranovsky in his proof \cite{Bar:2001} of the irreducibility of $C_2(\N(\mathfrak{gl}_n))$ for ${\rm char}(k)=0$ or ${\rm char}(k)\geq n$.
Baranovsky's result was extended to ${\rm char}(k)\geq \frac{n}{2}$ by Basili \cite{Bas}, and to arbitrary characteristic by Premet \cite{Pr:2003}, as a special case of equidimensionality of the (second) nilpotent commuting variety of the Lie algebra of a reductive algebraic group.
For larger values of $r$, work of the second and third authors shows that $C_r(\N(\mathfrak{gl}_n))$ is irreducible for $n\le 3$ and any $r$, and is reducible for all $r, n\ge 4$ \cite{N:2014}, \cite{NS}.
Apart from these reducibility results, little is known about these varieties for general $r$ and $n$.

%This paper is a contribution to the extensive literature on commuting varieties and their generalisations, e.g. \cite{Knut:2005}, \cite{GG:2006}, \cite{Z:2010}, \cite{Pan:1994}, \cite{Bu:09}.

In this paper, we study the dimensions of the ordinary and nilpotent commuting varieties for $\g=\mathfrak{gl}_n$.
The main initial idea is the following: let ${\mathfrak m}$ be a commutative nil (i.e. consisting of nilpotent elements) subalgebra of $\g$ of maximal dimension.
It is clear that for $r$ large enough, the dimension of the  linear span of a general $r$-tuple from $G\cdot \mathfrak{m}^r$ is not smaller than the dimension of the linear span of any $r$-tuple from $C_r(\N(\g))$.
Since this is an open condition and there are, up to conjugacy, at most two commutative nil subalgebras of $\g$ of maximal dimension, the subvariety $\overline{G\cdot{\mathfrak m}^r}$ is an irreducible component of $C_r(\N(\g))$ of dimension $r\dim{\mathfrak m}+\dim G-\dim N_G({\mathfrak m})$.
For $\g$ of type $A$, the commutative nil subalgebras of maximal dimension were studied in \cite{CFP}; the remaining types have been dealt with in the subsequent paper \cite{PS}, generalizing results of Malcev in characteristic zero \cite{Mal}.
Subsets $G\cdot \mathfrak{m}^r$ for various classical $\g$ were also investigated by the second author in \cite{Ngo:2014}.
Most of our effort in the first part of this paper will be directed towards showing that (under mild conditions on $r$) no irreducible component of $C_r(\N(\mathfrak{gl}_n))$ can have dimension greater than $\dim (G\cdot{\mathfrak m}^r)$.
As a consequence, we obtain (see Theorem \ref{precise_dim_commuting_nilpotents}):

\begin{theo}
Assume ${\rm char}\, k\neq 2,3$, $n\geq 4$, and $r\geq 7$.
Let $G=\GL_n$ and $\g=\mathfrak{gl}_n$.
Then
\[ \dim C_r(\N(\g))= (r+1)\lfloor \frac{n^2}{4}\rfloor,\]
and the irreducible components of maximal dimension are the subsets of the form $G\cdot{\mathfrak m}^r$ where ${\mathfrak m}$ is a commutative nil subalgebra of maximal dimension; there is one such component if $n$ is even, and two if $n$ is odd.
\end{theo}

We can apply this result to determine the irreducible components of maximal dimension of the $r$-th ordinary commuting varieties of $\mathfrak{gl}_n$ and $\mathfrak{sl}_n$, for the same range of values of $r$ and $n$ (see Cor. \ref{precise_dim_commuting_matrices} and Cor. \ref{precise_dim_commuting_matrices_sl}).

\begin{coro}
Assume ${\rm char}\, k\neq 2,3$, $n\geq 4$ and $r\geq 7$.
Exclude the case $(n,r)=(4,7)$.
Then $\dim C_r(\mathfrak{gl}_n)=(r+1)\lfloor \frac{n^2}{4}\rfloor + r$ and $\dim C_r(\mathfrak{sl}_n)=\dim C_r(\N(\mathfrak{gl}_n))$ unless ${\rm char}\, k|n$, in which case $\dim C_r(\mathfrak{sl}_n)=\dim C_r(\mathfrak{gl}_n)$.
\end{coro}

%(In the latter case, assuming $p\neq n$ and $(n,r)\neq (4,7)$, these are exactly the same as those identified in the Main Theorem.)

In a subsequent paper \cite{LNS}, we will apply a similar analysis to the $r$-th nilpotent commuting variety of the symplectic Lie algebra $\mathfrak{sp}_{2n}$.
The dimension is $(r+1)n(n+1)/2$ for large enough $n$ and $r$.
Note that in this case ${\mathfrak m}$ has dimension $n(n+1)/2$, and this also equals the codimension (in ${\rm Sp}_{2n}$) of the normalizer of ${\mathfrak m}$.
In general we make the following (see \cite{PS} for background on the variety ${\mathbb E}(d,\g)$):

\begin{conjecture}
Let $\g={\rm Lie}(G)$ where $G$ is a simple algebraic group.
For $r$ sufficiently large, we have
$$\dim C_r(\N(\g))=dr+m$$
where $d$ is the maximum dimension of an elementary subalgebra of $\g$ and $m=\dim {\mathbb E}(d,\g)$ is the dimension of the projective variety of all $d$-dimensional elementary subalgebras of $\g$.
\end{conjecture}

This conjecture should be seen in the context of Friedlander's support varieties for rational representations, see \cite{Friedlander-rational}; the formula $dr+m$  can be thought of as (conjecturally, in general) the asymptotic dimension of the topological space $V(G)$ of 1-parameter subgroups of $G$.

\subsection{Complexity of modules over Frobenius kernels}
After establishing our main theorem, we apply our results on nilpotent commuting varieties to cohomology of infinitesimal subgroups of $G$. Assuming $p>0$, let $F_r:G\to G$ be the $r$-{th} Frobenius morphism of $G$, defined on matrices as $(a_{ij})\mapsto(a^{p^r}_{ij})$.
Then the $r$-{th} \emph{Frobenius kernel} of $G$, denoted by $G_{(r)}$, is defined to be the scheme-theoretic kernel of $F_r$.
Note that $G_{(r)}$ is an {\it infinitesimal} group scheme, that is, it has only one rational point over the base field $k$.
The group structure comes into play when considering points over arbitrary $k$-algebras, or (equivalently) in the Hopf algebra structure of the (dual of the) finite-dimensional coordinate ring. (See \cite{Jan:2003} for an account of the background theory.)

The theory of support varieties for finite groups, and later for finite group schemes, was motivated by the idea that geometric methods might help to shed light on their (modular) representation theory.
In our case, one considers the even cohomology ring $\opH^{ev}(G_{(r)},k)$, which turns out to be commutative and finitely generated over $k$ \cite[Theorem 1.1]{FS:1997}, so its maximal ideal spectrum is an affine variety, the {\em support variety} of the trivial $G_{(r)}$-module.
By \cite[Theorem 5.2]{SFB2:1997} and \cite[Lemma 1.7]{SFB1:1997} this variety is homeomorphic to the $r$-th commuting variety $C_r(\N_{[p]}(\g))$, where $\N_{[p]}(\g)$ is the set of matrices in $\g$ with $p$-th power zero.
We therefore obtain, from our results on $C_r({\mathcal N}(\g))$, the Krull dimension of the ring $\opH^{ev}(G_{(r)},k)$ for $G=\GL_n(k)$. 
This dimension is the \emph{complexity} of the trivial module over $G_{(r)}$.
More generally, to an arbitrary $G_{(r)}$-module $M$ one associates a certain subvariety of $C_r(\N_{[p]}(\g))$; the dimension of this subvariety is the complexity of $M$.
In Theorem \ref{maxcomplex}, we deduce a criterion on the highest weight (assuming $p{>} n^3/4$) for the restriction to $G_{(r)}$ of a simple $G$-module to have maximal complexity.
Up to a point, this reduces the determination of the complexity of a given simple $G_{(r)}$-module to the study of the support varieties of the various simple restricted $\g$-modules arising from the $p$-adic decomposition of the highest weight.
%We remark that if the Lusztig Character Formula holds for all restricted dominant weights (a rather strong condition) then results of Drupieski-Nakano-Parshall and Sobaje allow us to identify exactly those simple modules $L(\lambda)|_{G_{(r)}}$ (where $\lambda$ is $p^r$-restricted) which have the same complexity as the trivial module.

Inspired by the results outlined in the previous paragraph, at the end of the paper we explore further a connection between the complexities of Frobenius kernels and finite groups of Lie type.
To explain this, let $G(\mathbb F_{p^r})$ be the group consisting of all $\mathbb F_{p^r}$-rational points of $G$, a finite Chevalley group.
Similarly to $G_{(r)}$, these subgroups inherit certain cohomological properties from the ambient algebraic group $G$.
The complexity theory of $G(\mathbb F_{p^r})$ is well-known from work of Quillen et al, see e.g. \cite[Ch. 5]{bensoncoh2} for further details.
The complexity of the trivial module for any finite group $\Gamma$ is equal to the $p$-rank, i.e. the maximal rank of an elementary abelian $p$-subgroup of $\Gamma$.
The $p$-ranks of the finite Chevalley groups are known, see e.g. \cite[S 15.4]{Hum}.
In particular, if $G=\SL_n$ then $$c_{G(\mathbb F_{p^r})}(k) = r\lfloor \frac{n^2}{4}\rfloor$$
which (by our main Theorem) equals $\tfrac{r}{r+1}c_{G_{(r)}}(k)$ for $r\geq 7$ and $n\geq 4$.
In the final section of the paper, we prove the inequality $$c_{G(\mathbb F_{p^r})}(M)\leq \frac{r}{r+1}c_{G_{(r)}}(M)$$ for an arbitrary finite-dimensional rational module $M$ over a simple algebraic group $G$ in good positive characteristic (see Theorem \ref{complexity bound}).
This significantly generalizes a result of Lin and Nakano \cite[Thm. 3.4(b)]{LN} (which considers the case $r=1$) and provides a new proof of the result of Drupieski \cite[Thm. 2.3]{Dru} that $M|_{G({\mathbb F}_{p^r})}$ is projective if $M|_{G_{(r)}}$ is.

\subsection{Organization of the paper}

In \S \ref{gradingsec} we collect some general results on nilpotent orbits and involutions which will be useful in the rest of the paper.
The main results concerning the nilpotent commuting variety are established in \S \ref{bounding}.
We first consider a certain subset $C'(x)$ of $C_r(\N(\g))$ related to the commuting variety of $\z_\g(x)$, where $x$ is a nilpotent element of $\g$ with no Jordan blocks of order greater than 4.
(See \S \ref{strategy} for the definition of $C'(x)$.)
In this special case, we analyse the subset $C'(x)$ in some detail, proving with the aid of various technical results from \S \ref{technical} that $\dim C'(x)$ is no greater than $(r+1)\lfloor \frac{n^2}{4}\rfloor$.
(To preserve the sanity of the reader, details of some computations are deferred to the Appendix.)
Once this case is settled, we can then prove the required inequality for an arbitrary $x$ by an induction argument.
Taken on their own, the technical results in \S \ref{technical} are unlikely to be of much interest, so a reader who does not want to get bogged down in the details might prefer to go straight to \S \ref{bounding}, only occasionally referring to the earlier sections for the auxiliary results.
Finally, \S \ref{complexity} deals with the connections with cohomological support and complexity of modules over finite group schemes.

\subsection{Acknowledgments}
The first and second authors were supported by an EPSRC grant, reference EP/K022997/1, which also paid for the third author to visit Lancaster to discuss this research project at its inception. The third author is partially supported by Slovenian research agency grants P1-0222 and N1-0103.
The authors would like to thank Chris Bendel, Dan Nakano, Paul Sobaje, Julia Pevtsova and Eric Friedlander for useful discussions. 

\subsection{Notation}
\begin{enumerate}
\item $k$ is an algebraically closed field of characteristic not equal to 2.
\item ${\rm char}\, k$ denotes the characteristic of $k$.
\item $\overline{V}$ denotes the Zariski closure of the set $V$ (the ambient space being understood).
\item $G$ is a linear algebraic group over $k$ (reductive unless otherwise stated), and $\g=\Lie(G)$.
\item $M_{s\times t}$ is the space of all $s\times t$ matrices.
\item $I_n$ is the identity $n\times n$ matrix.
\item $(-)^T$ denotes the transpose operation.
\item The adjoint action of $g\in G$ on $x\in\g$ is denoted $g\cdot x=\Ad g(x) = gxg^{-1}$. 
\item $\mathcal{O}_x=G\cdot x$ the orbit of an element $x$ in $\mathfrak{g}$.
\item $x\mapsto x^{[p]}$ is the $p$-operation on the restricted Lie algebra $\g$, almost always the $p$-th power of matrices.
\item $\N(\g)=\{x\in\g: x~\text{is~nilpotent}\}$ and when ${\rm char}\, k=p>0$, $\N_{[p]}(\g)=\{ x\in\g: x^{[p]}=0\}$, the restricted nullcone of $\g$. One always has $\N_{[p]}(\g)\subseteq\N(\g)$, equality occuring when $p\ge h$, the Coxeter number of $G$.
\item $\N(\g)/{G}$ denotes the (finite) set of $G$-orbits in $\N(\g)$.
\item $N_{G}(S)=\{g\in G:g\cdot S\subseteq S\}$, the normalizer in $G$ of the set $S$.
\item $Z_G(x)$ is the centralizer of $x$ in $G$ (where $x$ is in $G$ or $\g$).
Similarly, $\z_\g(x)$ is the centralizer of $x$ in $\g$.
%\item $\Omega_n$ denotes the space of symmetric $n\times n$ matrices.

%\item For any positive integer $r$ and any subvariety $V$ of $\mathfrak{gl}_n$ define
%$$C_r(V)=\{(x_1,\ldots ,x_r)\in V^r;[x_i,x_j]=0\, \mathrm{for}\, \mathrm{all}\, i,j=1,\ldots ,r\}.$$
%$C_r(V)$ is called the commuting variety of $r$-tuples over $V$. We define also the nilpotent cone $\mathcal{N}_n=\{x\in \mathfrak{gl}_n;x^n=0\}$ of $\mathfrak{gl}_n$. 
%\item $\mathfrak{O}_2=\{ x\in\g: x^2=0\}$. Note that $\mathfrak{O}_2\subseteq\N_{[p]}$ for all $p>0$.
%\item $\frakm_n$: the maximal commutative algebra in $\frakgl_n$ consisting of all matrices
%\[
%\begin{pmatrix}
%  aI_{\lfloor\frac{n}{2}\rfloor} & A\\
%  0 & aI 
%\end{pmatrix}
%\]
%where $a\in k$, and $A$ (resp. $I$) is an (resp. the identity) $\frac{n}{2}\times\frac{n}{2}$ matrix if $n$ is even, otherwise it is an $\frac{n-1}{2}\times\frac{n+1}{2}$ matrix (resp. the identity $\frac{n+1}{2}\times\frac{n-1}{2}$ matrix).  It is then easy to see that  
%\begin{align*}\label{dimM}
%\dim \frakm_n =\lfloor\frac{n^2}{4}\rfloor+1.
%\end{align*}
%\item $\N(\frakm_n)$: the maximal nilpotent subalgebra of $\frakm_n$. Then $\N(\frakm_n)$ is obtained by putting $a=0$ in the above description of $\frakm_n$. Hence, $\dim \N(\frakm_n) =\lfloor\frac{n^2}{4}\rfloor$. Note also that $\N(\frakm_n)\subseteq\mathfrak{O}_2$ for all $n$.
\end{enumerate}

\section{Some background}\label{gradingsec}

In this section we collect some standard results on orbital varieties and involutions.

\subsection{Gradings of nilpotent centralizers and orbital varieties}

Let $G$ be a simple algebraic group or $\GL_n$, let ${\mathfrak g}=\Lie(G)$ and let $e$ be a nilpotent element of $\g$.
In characteristic zero, we can embed $e$ in an $\mathfrak{sl}_2$-triple $\{ h,e,f\}\subseteq\g$.
Then ${\mathfrak z}_\g(e)\subseteq\sum_{i\geq 0}\g(h;i)$ where $\g(h;i)=\{ x\in\g : [h,x]=ix\}$.
In positive characteristic, one has to be careful about the use of $\mathfrak{sl}_2$-triples because, for example, the grading of $\g$ according to $(\ad h)$-eigenspaces becomes an ${\mathbb F}_p$-grading rather than a ${\mathbb Z}$-grading.
Under mild conditions one can replace the machinery of $\mathfrak{sl}_2$-triples by {\it associated cocharacters}.
An associated cocharacter for $e$ is a cocharacter $\lambda:k^\times\rightarrow G$ such that:

 - $\lambda(t)\cdot e = t^2 e$ for all $t\in k^\times$;

 - ${\mathfrak z}_\g(e)\subseteq\sum_{i\geq 0}\g(\lambda;i)$, where $\g(\lambda;i)=\{ x\in \g : \lambda(t)\cdot x=t^i x \;\mbox{for all}\; t\in k^\times\}$;

 - there is a Levi subgroup $L$ of $G$ such that $e\in\Lie(L)$ is distinguished and $\lambda(k^\times)\subseteq L^{(1)}=(L,L)$.

If the characteristic of the base field is a {\it good prime} for $G$ \cite[\S {4}]{sands} then every nilpotent element $e$ of $\g$ has an associated cocharacter, and any two such are conjugate by an element of $Z_G(e)$ \cite{Premnil}.
(For the purposes of the present paper we only need to know that all primes are good in type $A$; only $2$ is a bad prime for the other classical types.)
For a nilpotent Jordan block of order $m$ in $G=\GL_m$, the standard choice of associated cocharacter is $\lambda:k^\times\rightarrow G$, $\lambda(t)={\rm diag}(t^{m-1},t^{m-3},\ldots , t^{-(m-1)})$.
For $G=\GL_n$ and an arbitrary nilpotent matrix in Jordan normal form, we can construct an associated cocharacter by defining $\lambda$ in each block separately.
We now have ${\mathfrak z}_\g(e)=\sum_{i\geq 0}{\mathfrak z}(e;i)$ where ${\mathfrak z}(e;i)=\g(\lambda;i)\cap{\mathfrak z}_\g(e)$. The same holds in any reductive algebraic group $G$ such that $(G,G)$ is simply connected and $p$ is good prime for $G$ \cite[Proposition 5.8]{Jan:2004}.
The subalgebra ${\mathfrak z}(e;0)$ is called the {\it reductive part} of ${\mathfrak z}_\g(e)$.
(It is the Lie algebra of the reductive group $Z_G(e)\cap Z_G(\lambda(k^\times))$.)

To any nilpotent orbit ${\mathcal O}$ in a classical Lie algebra, one associates a partition $[m^{a_m},\ldots, 1^{a_1}]$ given by the sizes of the Jordan blocks of an element of ${\mathcal O}$.
Let $e\in{\mathcal O}$.
It is now well-known (see \cite{Premnil} for a unified proof) that when $G$ is a reductive group satisfying the {\it standard hypotheses} \cite[2.9]{Jan:2004}, the set of nilpotent orbits as well as the data on dimensions and reductive parts are the same as for the corresponding complex simple Lie algebra, and $\Lie(Z_G(e))={\mathfrak z}_\g(e)$ for any nilpotent element $e$.
(The standard hypotheses are always satisfied for $G=\GL_n$.)
In particular, if $\g={\mathfrak{gl}}_n$ then ${\mathfrak z}(e;0)$ is isomorphic to $\sum_{i=1}^m\mathfrak{gl}_{a_i}$ and $\dim{\mathfrak z}_\g(e)=\sum_{i=1}^m (a_i+\ldots +a_m)^2$.
% if $\g=\mathfrak{sp}_{2n}$ then all $a_i$ for $i$ odd are even, ${\mathfrak z}(e;0)$ is isomorphic to $\sum_{i\,\mbox{\scriptsize{odd}}}\mathfrak{sp}_{a_i}\oplus\sum_{i\,\mbox{\scriptsize{even}}}\mathfrak{so}_{a_i}$, and $\dim{\mathfrak z}_\g(e)=\frac{1}{2}\sum_{i=1}^m (a_i+\ldots +a_m)^2+\frac{1}{2}\sum_{1\leq j\leq \lceil \frac{m}{2}\rceil} a_{2j-1}$ \cite[\S 13]{Carter}.
(In general, the dimension of the centralizer is equal to $\dim{\mathfrak g}(\lambda;0)+\dim{\mathfrak g}(\lambda;1)$.)
For later reference we will be interested in the subset ${\mathfrak z}_\g'(e)\subseteq{\mathfrak z}_\g(e)$ of elements $y$ such that $ke+ky\subseteq\overline{\mathcal O}$.

\begin{lemma}\label{redcentlem}
Let ${\mathcal O}$ be a non-zero nilpotent orbit in the classical Lie algebra $\g$ and let $e\in{\mathcal O}$ have associated partition $[m^{a_m},\ldots, 1^{a_1}]$.
Suppose the characteristic of $k$ is either zero or strictly greater than $m$.
Then ${\mathfrak z}_\g'(e)\subseteq\sum_{i\geq 1}{\mathfrak z}(e;i)$.
\end{lemma}

\begin{proof}
In type $D$ we can replace the special orthogonal group $\SO_{2n}$ by the full orthogonal group ${\rm O}_{2n}$.
This ensures that the condition $ke+ky\subseteq\overline{\mathcal O}$ is equivalent to the set of rank conditions: ${\rm rank} (\xi e+\eta y)^r\leq{\rm rank}\, e^r$ for all $r\geq 1$ and all $\xi ,\eta \in k$.
Hence we can assume that $\g=\mathfrak{gl}_n$ (since any classical Lie algebra has a standard embedding in some $\mathfrak{gl}_n$, and an associated cocharacter $\lambda$ for $e$ in $\g$ is also an associated cocharacter for $e$ in $\mathfrak{gl}_n$).
Let $y=y_0+y_1+\ldots\in{\mathfrak z}_\g'(e)$ where $y_i\in{\mathfrak z}(e;i)$.
Then any scalar multiple of $y$ is in $\z_\g'(e)$, and hence (applying $\Ad\lambda(t)$ to $(e,t^2y)$ and scaling by $t^{-2}$) we have $ke+k(y_0+ty_1+\ldots)\subseteq\overline{\mathcal O}$ for any $t\in k^\times$.
In other words, ${\mathfrak z}_\g'(e)$ is $\Ad\lambda(k^\times)$-stable.
The various rank conditions can be expressed via various determinants of submatrices of $(\xi e+\eta y)^r$, and therefore ${\mathfrak z}_\g'(e)$ is also a Zariski closed subset of ${\mathfrak z}_\g(e)$.
Hence $y_0\in{\mathfrak z}_\g'(e)$.
We wish to show that $y_0=0$.

We recall that ${\mathfrak z}(e;0)\cong\mathfrak{gl}_{a_m}\oplus\ldots \oplus\mathfrak{gl}_{a_1}$.
After conjugating if necessary, we may assume that $$e=\begin{bmatrix} \widetilde{J_m} & & & \\&\widetilde{J_{m-1}} & & \\&&\ddots & \\&&&\widetilde{J_1}
\end{bmatrix}\quad \mathrm{where}\quad \widetilde{J_i}=\left[
\begin{array}{cccc}0&I_{a_i}\\&0&\ddots\\&&\ddots&I_{a_i}\\&&&0
\end{array}
\right]\in \mathfrak{gl}_{ia_i}\, \mathrm{for}\, 1\le i\le m.$$
Then the reductive part of the centralizer is the set of all matrices of the form: $$\begin{bmatrix} \Delta_{m}(A_m) &  &  & \\  & \Delta_{m-1}(A_{m-1}) &  &  \\  &  & \ddots &  \\  &  &  & \Delta_{1}(A_1) \end{bmatrix}\quad \mathrm{where}\; A_i\in\mathfrak{gl}_{a_i}\; \mbox{and}\quad \Delta_j(B)=\begin{bmatrix} B & & &  \\  & B &  &  \\  &  & \ddots & \\ & & & B \end{bmatrix}.$$
It is now clear that $y_0$ of the above form belongs to ${\mathfrak z}_\g'(e)$ if and only if each $\Delta_i(A_i)$ belongs to ${\mathfrak z}_{\mathfrak{gl}_{ia_i}}'(\widetilde{J}_i)$.
To see that we must have $A_i=0$, we simply compute $(\Delta_i(A_i)+\widetilde{J}_i)^i$, noting that the submatrix in the top right-hand corner is $iA_i$.
By our assumption on the characteristic, the requirement $iA_i=0$ implies $A_i=0$.
\end{proof}

\begin{remark}
a) Continuing the argument, we can be more precise: if $y=y_0+y_1+\ldots$ is in ${\mathfrak z}_\g'(e)$, then $y_0=0$ and $y_1\in{\mathfrak z}_\g'(e)$.
(These are necessary but not in general sufficient conditions.)

b) The condition $p>m$ is necessary, as the following example \cite[Rk. 3.1(3)]{Pr:2003} shows.
Let $e\in\mathfrak{gl}_{2p}$ be nilpotent with associated partition $[p,p]$.
Then $G\cdot e={\mathcal O}$ is the (unique) maximal orbit in ${\mathcal N}_{[p]}(\mathfrak{gl}_{2p})$, that is, $\overline{\mathcal O}$ is the set of all $2p\times 2p$ matrices with $p$-th power zero.
In this case the reductive part of the centralizer of $e$ is isomorphic to $\mathfrak{gl}_2$.
Letting $e_0\in{\mathfrak z}(e;0)$ be a non-zero nilpotent, we see that $e_0^{[p]}=0$.
By standard facts about the $p$-operation on a restricted Lie algebra, we have $(ae_0+be)^{[p]}=a^pe_0^{[p]}+b^pe^{[p]}=0$, whence $ae_0+be\in\overline{\mathcal O}$ for any $a,b\in k$.
Thus $e_0\in\mathfrak{z}_\g'(e)$.
\end{remark}

Let $B$ be a Borel subgroup of $G$ with unipotent radical $U$, and let ${\mathfrak b}$, resp. ${\mathfrak u}$ be the Lie algebra of $B$, resp. $U$.
Given a nilpotent orbit ${\mathcal O}$, the intersection ${\mathcal O}\cap{\mathfrak u}$ is an {\it orbital variety}.
The following fact will be very useful:

\vspace{0.1cm}
 - {\it ${\mathcal O}\cap{\mathfrak u}$ is equidimensional of dimension $\frac{1}{2}\dim{\mathcal O}$}.

\vspace{0.1cm}
(See \cite[Thm. 1, 10.6 and Thm., 10.11]{Jan:2004}. The proof in \cite{Jan:2004} applies in characteristic zero or good positive characteristic, which suffices for our purposes.)
We now apply this to prove a useful fact about the dimensions of certain subsets of centralizers of nilpotent elements.

\begin{lemma}\label{orbitallem}
Let ${\mathfrak g}$ be a classical Lie algebra, let ${\mathcal O}$ be a nilpotent orbit with associated partition $[m^{a_m},\ldots, 1^{a_1}]$ and let $e\in{\mathcal O}$.
Suppose the characteristic of $k$ is either zero or greater than $m$.
Then $\dim{\mathfrak z}_\g'(e)\leq \frac{1}{2}\dim{\mathcal O}$.
\end{lemma}

\begin{proof}
By Lemma \ref{redcentlem}, ${\mathfrak z}_\g'(e)\subseteq\sum_{i\geq 1}\g(\lambda;i)$, which is the nilradical of a parabolic subalgebra of $\g$ and is therefore contained in the nilradical ${\mathfrak u}$ of a Borel.
Thus $${\mathfrak z}_\g'(e)\subseteq \overline{\mathcal O}\cap{\mathfrak u}=\bigcup_{{\mathcal O}'\subseteq\overline{\mathcal O}}{\mathcal O}'\cap{\mathfrak u}$$
which is a finite union of locally closed subsets of dimension at most $\frac{1}{2}\dim{\mathcal O}$, so we are done.
\end{proof}

For clarity we recall our standing assumption that $k$ is not of characteristic 2.

\begin{lemma}\label{yzlem}
Let $G=\GL_n$ and let $e\in{\mathcal N}$ be a square zero matrix of rank $s$.
Let $t=n-2s$.
Then $\dim{\mathfrak z}_\g'(e)=s(s+t)$.
\end{lemma}

\begin{proof}
By standard results on dimensions of nilpotent orbits, the dimension of the orbit of $e$ is $(2s+t)^2-(s+t)^2-s^2=2s^2+2st$.
Hence, by Lemma \ref{orbitallem} we have $\dim{\mathfrak z}_\g'(e)\leq s(s+t)$.
To show that we have equality, it only remains to exhibit a closed subset of ${\mathfrak z}_\g'(e)$ of dimension $s(s+t)$.
We can assume after conjugation that 
\begin{equation}\label{centeq}
e=\begin{bmatrix} 0 & 0 & I_s \\ 0 & 0 & 0 \\ 0 & 0 & 0 \end{bmatrix};\;\;\;\;\mbox{then}\;\;\; {\mathfrak z}_\g'(e)\supseteq \left\{\begin{bmatrix} 0 & y & w \\ 0 & 0 & 0 \\ 0 & 0 & 0 \end{bmatrix} : y\in M_{s\times t}, w\in\mathfrak{gl}_s\right\}
\end{equation}
which has dimension $s(s+t)$.
\end{proof}

\begin{remark}\label{CFPremark}
We note that the subset on the right-hand side of \eqref{centeq} is the nilradical of a maximal parabolic subalgebra of $\g$, denoted ${\mathfrak u}_{s,s+t}$ and identified in \cite{CFP} as a maximal commutative nil subalgebra of $\g$.
It was proved in \cite{CFP} that if $n=2m$ (resp. $n=2m+1$) then up to conjugacy ${\mathfrak u}_{m,m}$ is the unique, resp. ${\mathfrak u}_{m,m+1},{\mathfrak u}_{m+1,m}$ are the only, commutative nil subalgebra(s) of $\g$ of maximal dimension.
\end{remark}

\subsection{Involutions of reductive groups}

In this subsection we will state some standard results about involutions of (Lie algebras of) reductive groups.
These were established in characteristic zero in \cite{Kostant-Rallis:1971}, and in odd positive characteristic (under mild hypotheses) by the first author in \cite{L:2007}.
Let $G$ be reductive and let $\theta:G\rightarrow G$ be an automorphism of order 2.
Then $d\theta$ is an involution of ${\mathfrak g}$.
Let ${\mathfrak k}$, resp. ${\mathfrak p}$ denote the $(+1)$, resp. $(-1)$ eigenspace for $d\theta$ on $\g$.
Then $\g={\mathfrak k}\oplus{\mathfrak p}$ and ${\mathfrak p}$ is stable under the adjoint action of the identity component of the fixed point subgroup $K=(G^\theta)^\circ$; moreover, $\Lie(K)={\mathfrak k}$.

In the next section we will require some facts about the orbits of $K$ on ${\mathfrak p}$, summarized in the following proposition and established in \cite[Prop. 16, Thm. 1, Prop. 5 and Thm. 9]{Kostant-Rallis:1971} and \cite[Thm. 2.11, Cor. 2.10, Lemma 4.1, Thm. 4.9 and Thm. 5.1]{L:2007}.
We recall that a subspace of ${\mathfrak p}$ which is maximal among the commutative subspaces consisting of semisimple elements is called a {\it Cartan subspace}.

\begin{prop}\label{involutionsprop}
Let $G$, $\theta$, $K$, ${\mathfrak k}$, ${\mathfrak p}$ be as above, and assume that the characteristic of the ground field is either zero or odd and good for $G$.

a) The semisimple elements are dense in ${\mathfrak p}$, and the $K$-orbit of $x\in{\mathfrak p}$ is closed if and only if $x$ is semisimple.
Any semisimple element of ${\mathfrak p}$ is contained in a Cartan subspace, and any two Cartan subspaces of ${\mathfrak p}$ are $K$-conjugate.

b) In positive characteristic, assume further that G is separably isogenous to a group satisfying the standard hypotheses.
For any $x\in{\mathfrak p}$, the intersection $(G\cdot x)\cap{\mathfrak p}$ consists of finitely many $K$-orbits, each of dimension $\frac{1}{2}\dim G\cdot x$.
\end{prop}

\section{Technical lemmas}\label{technical}

This section contains estimates of the dimensions of various varieties which are related to the nilpotent commuting variety $C_r(\N(\mathfrak{gl}_n))$.

\begin{lemma}\label{yz=0,zy=0}
Let
\[ \mathcal{U}_{s,t}=\{(y,z)\in M_{s\times t}\times M_{t\times s}:yz=0,zy=0\}.\]
Then $\dim \mathcal{U}_{s,t}=st$.
\end{lemma}

\begin{proof}
Conjugation by the diagonal matrix $g=\begin{bmatrix} I_s & 0 \\ 0 & -I_t \end{bmatrix}$ is an involutive automorphism of $\GL_{s+t}$. The fixed point subgroup for this automorphism is $\GL_s\times \GL_t$ and the $(-1)$-eigenspace in $\mathfrak{gl}_{s+t}$ is the set of all matrices of the form $M=\begin{bmatrix} 0 & y \\ z & 0 \end{bmatrix}$ where $y\in M_{s\times t}$ and $z\in M_{t\times s}$. The conditions $yz=0,zy=0$ are equivalent to $M^2=0$, so $\mathcal{U}={\mathcal U}_{s,t}$ is isomorphic to the variety of all square zero matrices in the $(-1)$ eigenspace $\mathfrak{p}$. Let $m=\min\{s,t\}$. The conditions $yz=0,zy=0$ are equivalent also to $\mathrm{im}\, z\subseteq \ker y,\mathrm{im}\, y\subseteq \ker z$, which imply $\mathrm{rank}( M)=\mathrm{rank}(y)+\mathrm{rank}(z)\le m$. The variety $\mathcal{U}$ is therefore a subset of the closure of the nilpotent $\GL_{s+t}$-orbit $\mathcal{O}$ corresponding to the partition $[2^m,1^{s+t-2m}]$.
There are finitely many $\GL_{s+t}$-orbits in $\overline{\mathcal O}$, and by Proposition \ref{involutionsprop} each such orbit $\mathcal{O}'$ intersects $\mathfrak{p}$ in finitely many $\GL_s\times \GL_t$-orbits, each of dimension $\frac{1}{2}\dim \mathcal{O}'$. Therefore $\dim \mathcal{U}\le \frac{1}{2}\dim \mathcal{O}$. A short calculation shows that $\dim \mathcal{O}=2st$, hence $\dim \mathcal{U}\le st$. To show the equality observe that $\{(y,0):y\in M_{s\times t}\}$ is a subvariety of $\mathcal{U}$ of dimension $st$.

\end{proof}

\begin{remark}
Using results on orbit closures in symmetric spaces one can show that the variety $\mathcal{U}_{s,t}$ is equidimensional.
Its irreducible components are
\[ \{(y,z)\in \mathcal{U}_{s,t}:\rank(y)\le m,\rank(z)\le \min \{s,t\}-m\} \]
for $m=0,\ldots ,\min \{s,t\}$.
\end{remark}

\begin{lemma}\label{y_iz_j=0}
Let $r$ be a positive integer, $s,t$ non-negative integers, and let
\[\mathcal{Y}_{r,s,t}=\{(y_1,\ldots ,y_r,z_1,\ldots ,z_r)\in M_{s\times t}^r\times M_{t\times s}^r: y_iz_j=0, 1\le i,j\le r\}.\]
Then $\dim\mathcal{Y}_{r,s,t}\le rst+\lfloor\frac{t^2}{4}\rfloor$.
\end{lemma}

\begin{proof}
We note that $\GL_t$ acts on ${\mathcal Y}={\mathcal Y}_{r,s,t}$ by %right multiplication on the $y_i$ and left multiplication on the $z_i$.
 $$g\cdot (y_1,\ldots ,y_r,z_1,\ldots ,z_r)=(y_1g^{-1},\ldots ,y_rg^{-1},gz_1,\ldots ,gz_r).$$
Let $(y_1,\ldots, z_r)\in{\mathcal Y}$ and let $W$ be the subspace of $k^t$ spanned by the columns of the $z_i$.
Then $z_i\in W\otimes (k^s)^T$ and $y_i\in k^s\otimes W^\perp$, where $W^\perp\subseteq (k^t)^T$ is the subspace of all linear forms which kill $W$.
(Here we identify $M_{s\times t}$ with $k^s\otimes (k^t)^T$, and we identify $(k^t)^T$ with $(k^t)^*$ via left multiplication.)
For a fixed $W$ of dimension $m$, the space of such tuples is therefore of dimension $rsm+rs(t-m)=rst$.
Now any two subspaces of $k^t$ of dimension $m$ are conjugate by the action of $\GL_t$, so fixing a subspace $W_m$ of dimension $m$ for each $m\in\{ 0,\ldots ,t\}$, we obtain:
$${\mathcal Y}=\cup_{m=0}^t \GL_t\cdot \left( (k^s\otimes W_m^\perp)^r\oplus (W_m\otimes (k^s)^T)^r\right).$$
Since the stabilizer of $W_m$ is a maximal parabolic subgroup of $\GL_t$ of dimension $t^2-tm+m^2$, we therefore have $\dim{\mathcal Y}\leq rst+{\rm max}_{0\leq m\leq t}\, m(t-m) = rst+\lfloor \frac{t^2}{4}\rfloor$.
\end{proof}

\begin{lemma}\label{y_iz_j=y_jz_i}
Let $r$ be a positive integer, $s,t$ nonnegative integers and let
\[ \mathcal{W}_{r,s,t}=\{(y_1,\ldots ,y_r,z_1,\ldots ,z_r)\in M_{s\times t}^r\times M_{t\times s}^r:y_iz_j=y_jz_i, ~ 1\le i,j\le r\}.\]
Then $\dim \mathcal{W}_{r,s,t}\le (r+1)st+\lfloor \frac{t^2}{2}\rfloor$.
Moreover,  for all $r\geq 2$ we have
\[\dim \mathcal{W}_{r,s,1}=
%\max\{rs,2s+r-1\}=
\left\{ \begin{array}{cl} rs & \mbox{if $s\geq 2$ and $r\geq 3$,} \\ 2s+1 & \mbox{if $s\geq 2$ and $r=2$,} \\ r+1 & \mbox{if $s=1$,} \\ 0 & \mbox{if $s=0$.}\end{array}\right.\]
In particular, $\dim \mathcal{W}_{r,s,1}\le rs+1$ for each $r\ge 2$ and each $s\ge {0}$.
\end{lemma}

\begin{proof}
The lemma clearly holds if $s=0$ or $t=0$, so we assume that $s,t\ge 1$.
We start with the following observation. A short computation shows that the action of $\GL_r$ on $M_{s\times t}^r\times M_{t\times s}^r$ defined by
$$(a_{ij})\bullet (y_1,\ldots ,y_r,z_1,\ldots ,z_r)=\left(\sum _{i=1}^ra_{1i}y_i,\ldots ,\sum _{i=1}^ra_{ri}y_i,\sum _{i=1}^ra_{1i}z_i,\ldots ,\sum _{i=1}^ra_{ri}z_i\right)$$
stabilizes the variety $\mathcal{W}_{r,s,t}$.
Since $\GL_r$ is connected, it also stabilizes each irreducible component of $\mathcal{W}_{r,s,t}$.
Therefore the subset $\mathcal{W}'_{r,s,t}$ of all $(2r)$-tuples $(y_1,\ldots ,y_r,z_1,\ldots ,z_r)\in \mathcal{W}_{r,s,t}$ such that {\it the rank of $y_1$ is not less than the rank of any linear combination of $y_1,\ldots ,y_r$} (*) is dense.

We first examine the case $t=1$. Let $\mathcal{C}$ be an irreducible component of ${\mathcal{W}_{r,s,1}}$.
If there exists a $(2r)$-tuple in $\mathcal{C}$ with $y_1\ne 0$ and $z_1\ne 0$, then the set of such elements is dense in ${\mathcal C}$ and the relations $y_1z_i=y_iz_1,2\le i\le r$ imply that for each $2\le i\le r$ there exists $\lambda _i\in k$ such that $(y_i,z_i)=\lambda _i(y_1,z_1)$.
It follows that ${\mathcal C}$ is the closure of the image of a morphism ${\mathbb A}^{s}\times {\mathbb A}^s\times{\mathbb A}^{r-1}\rightarrow {\mathcal W}_{r,s,1}$ that is injective on the open dense subset $(\mathbb{A}^s\backslash \{0\})\times (\mathbb{A}^s\backslash \{0\})\times \mathbb{A}^{r-1}$, hence $\dim \mathcal{C}=2s+r-1$.
On the other hand, if $y_1 =0$ or $z_1=0$ for each $(2r)$-tuple in $\mathcal{C}$, then by considering the $\GL_r$-action we see that we have either $y_i=0$ for $1\le i\le r$ or $z_i=0$ for $1\le i\le r$ in each $(2r)$-tuple in $\mathcal{C}$, yielding $\dim \mathcal{C}=rs$.
The description of the dimension of ${\mathcal W}_{r,s,1}$ follows by considering the difference $rs-(2s+r-1)=(r-2)(s-1)-1$.

Now we prove the general statement by induction on $s$ and $t$. We have already proved that $\dim \mathcal{W}_{r,1,1}=r+1$, therefore we assume that $s>1$ or $t>1$ and that $\dim \mathcal{W}_{\rho,\sigma,\tau}\le (\rho+1)\sigma\tau+\frac{\tau^2}{2}$ for all positive integers $\rho$ and nonnegative integers $\sigma$ and $\tau$ satisfying $\sigma\le s$ and $\tau<t$ or $\sigma<s$ and $\tau\le t$. 
For $0\leq m\leq \min\, \{ s,t\}$ let ${\mathcal W}_{(m)}$ be the set of all tuples $(y_1,\ldots ,z_r)\in{\mathcal W}_{r,s,t}$ with ${\rm rank}\, y_1=m$.
Clearly ${\mathcal W}_{(m)}$ and the intersection ${\mathcal W}'_{(m)}={\mathcal W}_{(m)}\cap{\mathcal W}'_{r,s,t}$ are quasi-affine.
Moreover, ${\mathcal W}'_{r,s,t}$ is the disjoint union of the ${\mathcal W}'_{(m)}$ and so it will suffice to prove that $\dim {\mathcal W}'_{(m)}\leq (r+1)st+\frac{t^2}{2}$ for all $m$.

Using the theorem on the dimension of fibres for the projection
\begin{align*} 
\pi_{r+1}:~~~~~~~~~~~~~~ &\mathcal{W}'_{(m)}~~~~~~~~~~~~~~\to~~~~~ M_{t\times s} \\
(y_1,\ldots ,y_r,&z_1,\ldots ,z_r) \mapsto z_1,
\end{align*}
%Applying the Krull's Principal ideal theorem (see e.g. \cite[Theorem 10.2]{E}) to the coordinate ring $k[\mathcal{W}']$ and the ideal generated by all entries of $z_1$
one gets $\dim \mathcal{W}'_{(m)}\le st+\dim \mathcal{Z}_m$ where
$$\mathcal{Z}_m={\mathcal W}'_{(m)}\cap \{ (y_1,\ldots ,y_r,0,z_2,\ldots,z_r) : y_i\in M_{s\times t},z_i\in M_{t\times s}\}.$$
(We remark that every irreducible component of ${\mathcal W}'_{(m)}$ intersects non-trivially with ${\mathcal Z}_m$, since if $(y_1,\ldots ,z_r)\in{\mathcal W}'_{(m)}$ then clearly $(y_1,\ldots ,y_r,\xi z_1,\ldots, \xi z_r)\in{\mathcal W}'_{(m)}$ for all $\xi\in k$.)
%\[
%\mathcal{Z}= \Big\{ (y_1,\ldots ,y_r,z_2,\ldots ,z_r)\in M_{s\times t}^r\times M_{t\times s}^{r-1} :y_1z_i=0~,~  %y_iz_j=y_jz_i~,~ 2\le i,j\le r~;~ \text{(*)~is~satisfied}\Big\}.
%\]

For any tuple $(y_1,\ldots ,y_r,0,z_2,\ldots ,z_r)\in \mathcal{Z}_m$ there exist bases of $k^s$ and $k^t$ with respect to which
\[ y_1 =\left[
\begin{array}{cc}I_m&0\\0&0
\end{array}
\right],
\]
therefore $\mathcal{Z}_m=(\GL_s\times \GL_t)\cdot \mathcal{V}_m$ where $\mathcal{V}_m=\left\{(y_1,\ldots ,y_r,0,z_2,\ldots ,z_r)\in \mathcal{Z}_m;y_1=\left[
\begin{array}{cc}I_m&0\\0&0
\end{array}
\right]\right\}$ and  the actions of $\GL_s\times \GL_t$ on $M_{s\times t}$ and on $M_{t\times s}$ are respectively defined by $(g,h)\cdot y=gyh^{-1}$ and $(g,h)\cdot z=hzg^{-1}$ for all $(g,h)\in \GL_s\times \GL_t, y\in M_{s\times t}, z\in M_{t\times s}$. It follows that
$$\dim \mathcal{Z}_m = \dim (\GL_s\times \GL_t)\cdot \left[
\begin{array}{cc}I_m&0\\0&0
\end{array}
\right]+\dim \mathcal{V}_m=m(s+t-m)+\dim \mathcal{V}_m.$$

We now examine $\mathcal{V}_m$.
Since $y_1z_i=0$ for each $2\le i\le r$, we obtain 
\[z_i=\left[
\begin{array}{cc}0&0\\z_i'&z_i''
\end{array}
\right]\] for some $z_i'\in M_{(t-m)\times m}$ and $z_i''\in M_{(t-m)\times (s-m)}$ for $2\le i\le r$. On the other hand, the condition (*) implies that for each $2\le i\le r$ one has 
\[ y_i=\left[
\begin{array}{cc}y_i'&y_i''\\y_i'''&0
\end{array}
\right]\] for some $y_i'\in \mathfrak{gl}_m$, $y_i''\in M_{m\times (t-m)}$ and $y_i'''\in M_{(s-m)\times m}$. Moreover, the conditions $y_iz_j=y_jz_i$ yield $y_i''z_j'=y_j''z_i'$ for all $2\le i,j\le r$, therefore
\begin{multline*}
\mathcal{V}_m\subseteq %(GL_s\times GL_t)\cdot
\bigg\{\left(\left[
\begin{array}{cc}I_m&0\\0&0
\end{array}
\right] ,\left[
\begin{array}{cc}y_2'&y_2''\\y_2'''&0
\end{array}
\right] ,\ldots ,\left[
\begin{array}{cc}y_r'&y_r''\\y_r'''&0
\end{array}
\right], 0 ,\left[
\begin{array}{cc}0&0\\z_2'&z_2''
\end{array}
\right] ,\ldots ,\left[
\begin{array}{cc}0&0\\z_r'&z_r''
\end{array}
\right]\right):\\
\quad\quad%\quad\quad\quad\quad\quad\quad\quad\quad\quad\quad\quad\quad\quad\quad\quad\quad
 y'_i\in \mathfrak{gl}_m,y_i'''\in M_{(s-m)\times m},z_i''\in M_{(t-m)\times (s-m)},~2\le i\le r,%\\
(y_2'',\ldots ,y_r'',z_2',\ldots ,z_r')\in \mathcal{W}_{r-1,m,t-m} \bigg\}.
\end{multline*}
%where the actions of $GL_s\times GL_t$ on $M_{s\times t}$ and on $M_{t\times s}$ are respectively defined by $(g,h)\cdot y=gyh^{-1}$ and $(g,h)\cdot z=hzg^{-1}$ for all $(g,h)\in GL_s\times GL_t, y\in M_{s\times t}, z\in M_{t\times s}$. %Clearly $\dim\mathcal{Z}_0=(r-1)st$ and $\dim \mathcal{Z}_t=(r-1)st$.
Using the induction hypothesis we have
%\begin{align*}
$$\dim \mathcal{V}_m \le %\dim\left( (GL_s\times GL_t)\cdot \left[
%\begin{array}{cc}I_m&0\\0&0
%\end{array}
%\right]\right) +
(r-1)(m^2+(s-m)m+(t-m)(s-m))%\\~~&~~
+rm(t-m)+\frac{(t-m)^2}{2}{=(r-1)st+\frac{t^2-m^2}{2}},%\\
$$
so
$$\dim \mathcal{Z}_m\le (r-1)st+\frac{t^2-m^2}{2}+m(s+t-m)\le rst+\frac{t^2}{2}$$
%=\dim \left((GL_s\times GL_t)\cdot \left[
%\begin{array}{cc}I_m&0\\0&0
%\end{array}
%\right]\right) +(r-1)st+\frac{t^2-m^2}{2}\le rst+\frac{t^2}{2}
%\end{align*}
%The orbit $(GL_s\times GL_t)\cdot \left[
%\begin{array}{cc}I_m&0\\0&0
%\end{array}
%\right]$ consists of all $s\times t$ matrices of rank $m$, therefore its dimension is equal to $m(s+t-m)$ by \cite[Exercise %10.10]{E}, and
%\begin{align*}
%\dim \mathcal{Z}_m &\le (r-1)st+t^2-tm+sm+m^2\\
%&=(r-1)st+s^2+t^2+m(m-t)+s(m-s)\\
%& \le (r-1)st+s^2+t^2
%\end{align*}
for each $m\le\min \{s,t\}$. Hence, we obtain $\dim \mathcal{W}_{r,s,t}=\dim \mathcal{W}'_{r,s,t}\le st+\dim \mathcal{Z}_m\le (r+1)st+\frac{t^2}{2}$.
\end{proof}

\begin{lemma}\label{Vmlc}
Let $c,m,l$ be non-negative integers such that $2m+l\leq c$.
Then the subset $${\mathcal V}_{c,m,l}=\{ u\in\mathfrak{gl}_c : {\rm rank}(u)=c-m-l, {\rm rank}(u^2)=c-2m-l\}$$
has dimension $c^2-(2m^2+2ml+l^2)$.
%a) ${\mathcal V}_{c,m,l}$ is a locally closed subset of $\mathfrak{gl}_c$ of dimension $c^2-(2m^2+2ml+l^2)$.
\end{lemma}

\begin{proof}
It is easy to see that ${\mathcal V}_{c,m,l}$ %and ${\mathcal V}_{c,m,l}^{sp}$ are
is locally closed subset since the rank conditions ${\rm rank}(u)\leq a$ and ${\rm rank}(u^2)\leq b$ are closed conditions for any $a,b\in{\mathbb N}_0$.
Any $u\in{\mathcal V}_{c,m,l}$ is $\GL_c$-conjugate to a matrix in Jordan normal form $\begin{bmatrix} u' & 0 \\ 0 & u''\end{bmatrix}$ where $u''$ is invertible and $u'$ is nilpotent, with $(m+l)$ Jordan blocks, exactly $l$ of which are of order 1.
Given $\lambda=[\lambda_1\geq\ldots\geq \lambda_m]$ such that $\lambda_m>1$ and $n_\lambda=\sum_{i=1}^m\lambda_i+l\leq c$ we choose a nilpotent $n_\lambda\times n_\lambda$ matrix $u'_\lambda$ with Jordan blocks of sizes $\lambda_1,\ldots ,\lambda_m,1,\ldots ,1$.
Then the set $${\mathcal V}_{c,\lambda}:=\GL_c\cdot \left\{ \begin{bmatrix} u'_\lambda & 0 \\ 0 & u''\end{bmatrix} : u''\in\GL_{c-n_\lambda}\right\}$$ is Zariski constructible (and irreducible) and ${\mathcal V}_{c,m,l}$ is the (disjoint) union of these subsets, taken over all relevant partitions $\lambda$.
To determine dimensions we note that $\GL_c\cdot{\mathcal S}_\lambda$ is dense in ${\mathcal V}_{c,\lambda}$ where
$${\mathcal S}_\lambda=\left\{ \begin{bmatrix} u'_\lambda & 0 \\ 0 & u''\end{bmatrix} : u''\in\GL_{c-n_\lambda}\;\mbox{diagonal with distinct eigenvalues}\;\right\}$$ and that any $\GL_c$-orbit in ${\mathcal V}_{c,\lambda}$ has finite (possibly empty) intersection with ${\mathcal S}_\lambda$, by the uniqueness of the Jordan normal form.
The transpose of the partition $[\lambda,1^l]$ is $[m+l,m,\ldots]$, where there are non-zero terms after $m$ if and only if $\lambda_1>2$.
Since $\dim{\mathfrak z}_{\mathfrak{gl}_{n_\lambda}}(u_\lambda')$ is the sum of the squares of the parts of $[\lambda,1^l]^T$, it follows that $$\dim{\mathcal V}_{c,\lambda} = c^2-\dim N_{\GL_c}({\mathcal S}_\lambda)+(c-n_\lambda)=c^2-\dim {\mathfrak z}_{\mathfrak{gl}_{n_\lambda}}(u'_\lambda)=c^2-((m+l)^2+m^2+\ldots )$$
which is maximal precisely when $\lambda=[2^m]$.
The statement on the dimension follows.
\end{proof}

The following lemma plays a crucial role in establishing the dimension of $C_r(\N(\mathfrak{gl}_n))$, see the proof of Lemma \ref{GLblocksize4}.

\begin{lemma}\label{y,z,w,v,u}
Let $r>5$ be a positive integer and $a,b,c$ nonnegative integers. Let
$\mathcal{Z}_{r,a,b,c}$ be the subvariety of $M_{a\times c}^r\times M_{b\times c}^r\times M_{c\times a}^r\times M_{c\times b}^r\times \mathfrak{gl}_c^r$ consisting of all tuples $$(y_1,\ldots ,y_r,z_1,\ldots ,z_r,w_1,\ldots ,w_r,v_1,\ldots ,v_r,u_1,\ldots ,u_r)$$ satisfying the following equations:
$$y_iw_j=y_jw_i,z_iv_j=0,y_iu_j=y_ju_i,z_iu_j=z_ju_i,u_iw_j=u_jw_i,u_iv_j=u_jv_i$$
for each $1\le i,j\le r$. Then:

a) We have $\dim \mathcal{Z}_{r,a,b,0}=0$, $\dim{\mathcal Z}_{r,1,0,1}=r+2$ and $\dim{\mathcal Z}_{r,0,0,1}=r$; if $(a,b)\not\in\{ (1,0), (0,0)\}$ then $\dim \mathcal{Z}_{r,a,b,1}\le r(a+b)+1$.

b) If $c\geq 2$ then
$$\dim\mathcal{Z}_{r,a,b,c}\leq \left\{ \begin{array}{cc} \frac{a^2}{11}+c^2+2ac+2bc+(r-1)(c^2+\frac{(a+b)^2}{2}) & \mbox{if $c\geq a+b$,} \\
 \frac{a^2}{11}+c^2+2ac+2bc+(r-1)(\frac{c^2}{2}+(a+b)c) & \mbox{if $c\leq a+b$.}
\end{array}\right.$$
\end{lemma}

\begin{proof}
As in Lemma \ref{y_iz_j=y_jz_i}, the irreducible components of $\mathcal{Z}={\mathcal Z}_{r,a,b,c}$ are invariant under the action of $\GL_r$, hence the subset $\mathcal{Z}'$ of ${\mathcal Z}$ consisting of all tuples $(y_1,\ldots ,u_r)\in \mathcal{Z}$ such that % $\mathrm{rank}(u_1)=\max_i \mathrm{rank}(u_i)$ and $\mathrm{rank}(u_1^2)=\max_i \mathrm{rank}(u_i^2)$
{the rank of $u_1$ is not less than the rank of any linear combination of $u_1,\ldots ,u_r$ and the rank of $u_1^2$ is not less than the rank of the square of any linear combination of $u_1,\ldots ,u_r$}
 is dense.
In particular, $\dim \mathcal{Z}=\dim \mathcal{Z}'$.

As the case $c=0$ is trivial, we first consider the case $c=1$. Let $\mathcal{C}$ be any component of $\mathcal{Z}$. The condition $z_iv_j=0$ for all $1\le i,j\le r$ implies $z_1=\cdots =z_r=0$ or $v_1=\cdots =v_r=0$. By symmetry we can assume that the latter holds on $\mathcal{C}$. If there exists a $(5r)$-tuple in $\mathcal{C}$ such that $u_i\ne 0$ for some $i$, then $y_j=\frac{u_j}{u_i}y_i$, $z_j=\frac{u_j}{u_i}z_i$ and $w_j=\frac{u_j}{u_i}w_i$ for $1\le j\le r$, which yields $\dim \mathcal{C}=r+2a+b$. On the other hand, if $u_i=0$ for all $1\le i\le r$ and all tuples in $\mathcal{C}$, then Lemma \ref{y_iz_j=y_jz_i} implies $\dim \mathcal{C}\le r(a+b)+1$. This proves (a).
%, and also (b) for $c\le 1$.

In the rest of the proof we assume $c\ge 2$.
For all nonnegative integers $m$ and $l$ with $2m+l\le c$ we define $\mathcal{Z}_{m,l}$ as the subset of $\mathcal{Z}'$ consisting of all tuples $(y_1,\ldots ,u_r)$ satisfying $u_1\in{\mathcal V}_{c,m,l}$. (See Lemma \ref{Vmlc}.) This is clearly a locally closed set and the union of all $\mathcal{Z}_{m,l}$ is $\mathcal{Z}'$. Moreover, if $(y_1,\ldots ,u_r)\in \mathcal{Z}_{m,l}$, then the $\GL_r$-action gives us $u_i\in \overline{\mathcal{V}_{c,m,l}}$ for $i=2,\ldots ,r$.
We consider the projection \[\pi \colon \mathcal{Z}_{m,l}\to M_{a\times c}\times M_{b\times c}\times M_{c\times a}\times M_{c\times b}\cong {\mathbb A}^{2(a+b)c}\] defined by
$$\pi (y_1,\ldots ,y_r,z_1,\ldots ,z_r,w_1,\ldots ,w_r,v_1,\ldots ,v_r,u_1,\ldots ,u_r)=(y_1,z_1,w_1,v_1).$$
The preimage $\pi ^{-1}(0,0,0,0)$ intersects every irreducible component of ${\mathcal Z}_{m,l}$, since, for any $t\in k$ and any $(y_1,\ldots ,u_r)\in{\mathcal Z}_{m,l}$, we clearly have $(ty_1,\ldots ,tv_r,u_1,\ldots ,u_r)\in{\mathcal Z}_{m,l}$.
Hence $\dim \mathcal{Z}_{m,l}\le 2ac+2bc+\dim \pi ^{-1}(0,0,0,0)$.
Next, we consider the projection $\pi '\colon \pi ^{-1}(0,0,0,0)\to \overline{\mathcal{V}_{c,m,l}}^{r-1}$ defined by
$$\pi' (0,y_2,\ldots ,y_r,0,z_2,\ldots ,z_r,0,w_2,\ldots ,w_r,0,v_2,\ldots ,v_r,u_1,\ldots ,u_r)=(u_2,\ldots ,u_r).$$
As above, we observe that $\pi^{-1}(0,0,0,0)$ is invariant under multiplying the last $r-1$ components by an arbitrary scalar, hence $\pi '^{-1}(0,\ldots ,0)$ intersects every irreducible component of $\pi^{-1}(0,0,0,0)$. Consequently, $\dim \mathcal{Z}_{m,l}\le 2ac+2bc+(r-1)(c^2-2m^2-2ml-l^2)+\dim \mathcal{W}_{m,l}$ where
\[\mathcal{W}_{m,l}=
\left\{
\begin{array}{cc}
& (y_2 ,..,y_r,z_2,..,z_r,w_2,..,w_r,v_2,..,v_r,u_1)\in M_{a\times c}^{r-1}\times M_{b\times c}^{r-1}\times M_{c\times a}^{r-1}\times M_{c\times b}^{r-1}\times \mathcal{V}_{c,m,l}:\\
& y_iu_1=0,z_iu_1=0,u_1w_i=0,u_1v_i=0,y_iw_j=y_jw_i,z_iv_j=0, 2\le  i,j\le r.
\end{array} \right\}
\]
For each tuple $(y_2,\ldots ,y_r,z_2,\ldots ,z_r,w_2,\ldots ,w_r,v_2,\ldots ,v_r,u_1)\in \mathcal{W}_{m,l}$, with respect to a basis of $k^c$, we can assume that there exists nonnegative integer $t\le c-2m-l$ such that 
\[ u_1=\left[
\begin{array}{ccc}u_1'&0&0\\0&0&0\\0&0&u_1''
\end{array}
\right]\]
for some invertible $t\times t$ matrix $u_1''$ and some nilpotent $(c-t-l)\times (c-t-l)$ matrix $u_1'$ in the Jordan canonical form which has $m$ Jordan blocks, all of them of order more than 1. Then
$$y_i=\left[
\begin{array}{ccc}y_i'&y_i''&0
\end{array}
\right] ,\quad z_i=\left[
\begin{array}{ccc}z_i'&z_i''&0
\end{array}
\right] ,\quad w_i=\left[
\begin{array}{c}w_i'\\w_i''\\0
\end{array}
\right] ,\quad v_i=\left[
\begin{array}{c}v_i'\\v_i''\\0
\end{array}
\right]$$
for $2\le i\le r$ where $y_i'\in M_{a\times (c-t-l)},y_i''\in M_{a\times l},z_i'\in M_{b\times (c-t-l)},z_i''\in M_{b\times l},w_i'\in M_{(c-t-l)\times a},w_i''\in M_{l\times a},v_i'\in M_{(c-t-l)\times b},v_i''\in M_{l\times b}$, the transposes of the rows of the matrices $y_i'$ and $z_i'$ belong to the $m$-dimensional kernel of $u_1'^T$ and the columns of the matrices $w_i'$ and $v_i'$ belong to the $m$-dimensional kernel of $u_1'$.
Therefore $y_i'w_j'=0$ and $z_i'v_j'=0$ for all $2\le i,j\le r$.
The conditions $y_iw_j=y_jw_i$ and $z_iv_j=0$ are then equivalent to $y_i''w_j''=y_j''w_i''$ and $z_i''v_j''=0$ for $2\le i,j\le r$.
Lemmas \ref{y_iz_j=0}, \ref{y_iz_j=y_jz_i} and \ref{Vmlc} now imply
\begin{align*}
\dim \mathcal{W}_{m,l} &\le \dim \mathcal{V}_{c,m,l}+2(r-1)(a+b)m+ral+\left\lfloor\frac{l^2}{2}\right\rfloor+(r-1)bl+\left\lfloor \frac{l^2}{4}\right\rfloor\\
&=c^2-2m^2-2ml+\left\lfloor \frac{l^2}{4}\right\rfloor -\left\lceil\frac{l^2}{2}\right\rceil +al+(r-1)(a+b)(2m+l),
\end{align*}
therefore
$$\dim \mathcal{Z}_{m,l}\le al+c^2+2ac+2bc-2m^2-2ml+\left\lfloor \frac{l^2}{4}\right\rfloor -\left\lceil \frac{l^2}{2}\right\rceil+(r-1)(c^2-2m^2-2ml-l^2+(a+b)(2m+l)).$$
Since $r\ge 6$, we clearly have
$$\left\lfloor \frac{l^2}{4}\right\rfloor\leq \frac{(2r-11)l^2}{4}=\frac{(r-1)l^2}{2}-\frac{9l^2}{4}.$$
Furthermore, $al-\left\lceil\frac{l^2}{2}\right\rceil-\frac{9l^2}{4}\leq al-\frac{11l^2}{4}\leq \frac{a^2}{11}$ and therefore $$al-\left\lceil\frac{l^2}{2}\right\rceil-\frac{9l^2}{4}+c^2+2ac+2bc-2m^2-2ml\le \frac{a^2}{11}+c^2+2ac+2bc.$$
On the other hand, since $2m+l\le c$, the expression
\begin{align*}
c^2-2m^2-2ml-\frac{l^2}{2}+(a+b)(2m+l) &=c^2-\frac{1}{2}(2m+l)^2+(a+b)(2m+l)
\end{align*}
is maximal if $2m+l=\min\{a+b,c\}$, so
\[ c^2-\frac{1}{2}(2m+l)^2+(a+b)(2m+l)\le 
\begin{cases}
c^2+\frac{(a+b)^2}{2}\quad &\mathrm{if~~} a+b\le c,\\
\frac{c^2}{2}+(a+b)c &\mathrm{if~~} a+b\ge c.
\end{cases}
\]
It follows that
\[ \dim \mathcal{Z'}=\max_{m,l}\{\dim\mathcal{Z}_{m,l}\}\le 
\begin{cases}
\frac{a^2}{11}+c^2+2ac+2bc+(r-1)\left(c^2+\frac{(a+b)^2}{2}\right) &\text{~if~~} a+b\le c,\\
\frac{a^2}{11}+c^2+2ac+2bc+(r-1)\left(\frac{c^2}{2}+(a+b)c\right) &\text{~if~~} a+b\ge c,
\end{cases}\]
hence completing our proof.
\end{proof}

\section{Dimension of the commuting variety}\label{bounding}

In this section, we use the estimates established in the previous two sections to get upper bounds for $\dim C_r(\N(\g))$.
When $r$ is large enough, we are able to exactly describe the dimension of $C_r(\N(\g))$ and the irreducible components of maximal dimension.
We assume from now on that the characteristic of $k$ is not equal to 2 or 3.
We first explain our general strategy.

\setcounter{subsection}{-1}

\subsection{General stategy}\label{strategy} Although we focus in this paper on the group $\GL_n$, the present strategy is applicable to any reductive algebraic group.
Recall that $\N(\g)$ is the nilpotent cone of $\g$.
For each element $e$ in $\N(\g)$, set $C(e)=\overline{G\cdot (e,C_{r-1}(\mathfrak{z}_\g(e)\cap \mathcal{N}(\g)))}$.
Then we have for all $r\ge 2$  $$C_r(\mathcal{N}(\g))\:=\bigcup _{\calO_e\in \mathcal{N}(\g)/G}C(e).$$
Since this union is finite, each irreducible component of $C_r(\mathcal{N}(\g))$ is an irreducible component of some $C(e)$ with $e\in \mathcal{N}(\g)$.
Generalizing an idea in \cite[Prop. 2.1]{Pr:2003}, there is an action of $\GL_r$ on $C_r(\mathcal{N}(\g))$ defined by
$$(a_{ij})\bullet (x_1,\ldots ,x_r)=\left(\sum _{i=1}^ra_{1i}x_i,\ldots ,\sum _{i=1}^ra_{ri}x_i\right)$$
which stabilizes each irreducible component of $C_r(\mathcal{N}(\g))$.
Thus each irreducible component of $C_r(\mathcal{N}(\g))$ is a subset of $C'(e)=\overline{G \cdot (e,C'_{r-1}(\mathfrak{z}_\g(e)))}$ for some $e\in \mathcal{N}(\g)$, where
\[ C'_{r-1}(\mathfrak{z}_\g(e))=\left\{(y_1,\ldots ,y_{r-1})\in C_{r-1}(\mathfrak{z}_\g(e)\cap\N(\g)): e+\sum _{i=1}^{r-1}ky_i\subseteq \overline{\mathcal{O}_e}\right\}.\]
We earlier (before Lemma \ref{redcentlem}) introduced the subset $\mathfrak{z}_\g'(e)$ of $\mathfrak{z}_\g(e)$ consisting of all $y$ such that $ke+ky\subseteq\overline{\mathcal{O}_e}$.
Clearly $C'_{r-1}(\mathfrak{z}_\g(e))\subseteq C_{r-1}({\mathfrak z}_\g'(e))$.
We can obtain an upper bound for $\dim C_r(\N(\g))$ by determining bounds for $\dim C'_{r-1}(\mathfrak{z}_\g(e))$ or $\dim C'(e)$, for $e$ belonging to each orbit in $\N(\g)$.
%We start by giving an upper bound for this dimension in several special cases.

\begin{proposition}\label{square_zero}
Let $x\in \mathcal{N}(\mathfrak{gl}_n)$ be a square zero matrix and let $s=\rank(x)$ and $t=n-2s$. Then for $r\ge 1$, we have
\[\dim C'_r(\mathfrak{z}_\g(x))=rs(s+t).\]
\end{proposition}

\begin{proof}
We have $C'_r({\mathfrak z}_\g(x))\subseteq C_r({\mathfrak z}_\g'(x))$ and therefore $\dim C'_r({\mathfrak z}_\g(x))\leq r\dim {\mathfrak z}'_\g(x)$.
By Lemma \ref{yzlem} we have $\dim C'_r({\mathfrak z}_\g(x))\leq rs(s+t)$.
To obtain equality, write $x=\begin{bmatrix} 0&0&I_s\\0&0&0\\0&0&0\end{bmatrix}$ as in the proof of Lemma \ref{yzlem}.
Then ${\mathfrak u}_{s,s+t}^r$ (see Remark \ref{CFPremark})
is a subset of $C'_r(\mathfrak{z}_\g(x))$ of dimension $rs(s+t)$, hence we have $\dim C'_r({\mathfrak z}_\g(x))\geq rs(s+t)$ too.
\end{proof}

\begin{corollary}\label{square_zeroC'(x)}
Let $x\in \mathcal{N}(\mathfrak{gl}_n)$ be a square zero matrix and let $s=\rank(x)\le \lfloor \frac{n}{2}\rfloor$ and $t=n-2s$. Then $\dim C'(x)=(r+1)s(s+t)$. In particular, if we assume further that $\rank(x)=\lfloor \frac{n}{2}\rfloor$, then 
\[ \dim C'(x)=(r+1)\left\lfloor \frac{n^2}{4}\right\rfloor. \]
Specifically, if $n=2m$ then $C'(x)=G\cdot {\mathfrak u}_{m,m}^r$ is irreducible; if $n=2m+1$ then $C'(x)=G\cdot {\mathfrak u}_{m+1,m}^r\cup G\cdot {\mathfrak u}_{m,m+1}^r$ is equidimensional.
\end{corollary}

\begin{proof}
The first statement follows immediately from Proposition \ref{square_zero} and the equality
\begin{equation}\label{dimC'(x)}
\dim C'(x)=n^2-\dim \mathfrak{z}_\g(x)+\dim C'_{r-1}(\mathfrak{z}_\g(x)).
\end{equation}
The statements about the maximal rank case can be deduced immediately from the fact that, if $x$ is in the form indicated in the proof of Prop. \ref{square_zero}, then $\z_\g'(x)$ equals ${\mathfrak u}_{m,m}$, resp. ${\mathfrak u}_{m+1,m}\cup{\mathfrak u}_{m,m+1}$ if $n=2m$, resp, $2m+1$. Note that each set $G\cdot {\mathfrak u}_{l,n-l}^r$ is indeed closed, since it is defined by equations $x_ix_j=0$, $1\le i,j\le r$, and the closed conditions that the rows of all $x_i$ span at most an $(n-l)$-dimensional space and their columns span at most an $l$-dimensional space.
\end{proof}

\begin{lemma}\label{rough_dim_commuting_nilpotents}
For all $n\ge 1$ and $r\ge 2$, we have $\dim C_r(\mathcal{N}(\mathfrak{gl}_n))\le (r+1)\frac{n^2-1}{3}$.
\end{lemma}

\begin{proof}
Let $x\in \mathcal{N}(\mathfrak{gl}_n)$ be a nilpotent matrix with associated partition $[m^{a_m},\ldots ,2^{a_2},1^{a_1}]$, where $a_1,\ldots ,a_m\ge 0$ and $a_m>0$.
It suffices to show that 
$\dim C'(x)\le (r+1)\frac{n^2-1}{3}$  for each such $x$,
which is equivalent to
\[ \dim C'_{r-1}(\z_\g(x))\le (r+1)\frac{n^2-1}{3}-\dim GL_n+\dim\z_\g(x)=\frac{r-2}{3}n^2+\dim{\mathfrak z}_\g(x)-\frac{r+1}{3}. \]
If $m\le 2$, then the inequality follows from Corollary \ref{square_zeroC'(x)}.
We now assume $m\ge 3$ and proceed by induction on $m$.
Write $x$ in the form indicated in the proof of Lemma \ref{redcentlem}.
Consider a tuple $(y_1,\ldots ,y_{r-1})$ in $C'_{r-1}(\mathfrak{z}_\g(x))$, where each $y_l$, for $1\le l\le r-1$, is of the form
$$y_l=\left[
\begin{array}{cccc}y_{mm}^{(l)}&y_{m,m-1}^{(l)}&\cdots&y_{m1}^{(l)}\\y_{m-1,m}^{(l)}&y_{m-1,m-1}^{(l)}&\cdots&y_{m-1,1}^{(l)}\\\vdots&\vdots&\ddots&\vdots\\y_{1m}^{(l)}&y_{1,m-1}^{(l)}&\cdots&y_{11}^{(l)}
\end{array}
\right]\quad \mathrm{where}\quad y_{ii}^{(l)}=\left[
\begin{array}{cccc}y_{ii1}^{(l)}&y_{ii2}^{(l)}&\cdots&y_{iii}^{(l)}\\0&y_{ii1}^{(l)}&\ddots&\vdots\\\vdots&\ddots&\ddots&y_{ii2}^{(l)}\\0&\cdots&0&y_{ii1}^{(l)}
\end{array}
\right]$$
for some matrices $y_{ii1}^{(l)},\ldots ,y_{iii}^{(l)}$ of appropriate size, and for $i\ne j$ the matrix $y_{ij}^{(l)}$ has the same form with some additional block rows of zeros at the end or some additional block columns of zeros at the beginning (see e.g. \cite[Proposition 14]{NS}).

Let 
\[ x'=\left[
\begin{array}{ccc}\widetilde{J_{m-1}}\\&\ddots\\&&\widetilde{J_1}
\end{array}
\right], \]
and $n'=n-ma_m$.
Let $\g'=\mathfrak{gl}_{n'}$.
{Assume first that $n'=0$, i.e. $n=ma_m$. Then $\dim \mathfrak{z}_{\g}(x)=ma_m^2$. Since $y_l$ is nilpotent for each $l$, $y_{mm1}^{(l)}$ is nilpotent for each $l$. Recall that $\dim\mathcal{N}(\mathfrak{gl}_{a_m})=a_m^2-a_m$. Then we have
$$\dim C_{r-1}'(\mathfrak{z}_\g(x))\le (r-1)(ma_m^2-a_m)\le (r-2)\frac{m}{3}ma_m^2+\dim \mathfrak{z}_\g(x)-(r-1)\le \frac{r-2}{3}n^2+\dim \mathfrak{z}_\g(x)-\frac{r+1}{3}$$
for all $r\ge 2$, $m\ge 3$ and $a_m\ge 1$.

Assume now that $n'>0$.}
By the induction hypothesis we have $$\dim C'_{r-1}({\mathfrak z}_{\g'}(x'))\le \frac{r-2}{3}{n'}^2+\dim{\mathfrak z}_{\g'}(x')-\frac{r+1}{3}.$$
We note that $\dim{\mathfrak z}_\g(x)-\dim{\mathfrak z}_{\g'}(x')=ma_m^2+2n'a_m=a_m(2n-ma_m)$.
Applying Krull's Principal ideal theorem to the coordinate ring of any irreducible component of $C'_{r-1}(\mathfrak{z}_\g(x))$ and its ideal generated by all entries of $y_{mi}^{(l)}$ and $y_{im}^{(l)}$ for $1\le i\le m$, $1\le l\le r-1$, we have
\begin{align*}
\dim C'_{r-1}(\mathfrak{z}_\g(x)) &\le (r-1)(\dim\z_\g(x)-\dim\z_{\g'}(x'))+\dim C'_{r-1}(\z_{\g'}(x'))\\
&\le (r-2)\left(\frac{{n'}^2}{3}+\dim\mathfrak{z}_\g(x)-\dim\mathfrak{z}_{\g'}(x')\right)+\dim{\mathfrak z}_\g(x)-\frac{r+1}{3}\\
&= \frac{r-2}{3}\left( (n-ma_m)^2+3a_m(2n-ma_m)\right)+\dim{\mathfrak z}_\g(x)-\frac{r+1}{3}\\
&=\frac{r-2}{3}n^2+\dim{\mathfrak z}_\g(x)-\frac{r-2}{3}(m-3)a_m(2n-ma_m)-\frac{r+1}{3}\\
&\le \frac{r-2}{3}n^2+\dim\z_\g(x)-\frac{r+1}{3}
\end{align*}
for all $m\geq 3$, as required.
\end{proof}

\begin{remark}\label{inductioninequality}
In the block form indicated in the above proof, the component of $y_l$ in the reductive part of the centralizer (see \S \ref{gradingsec}) is given by the submatrices $y_{ii1}^{(l)}$.
Since $y_l$ is nilpotent then in fact $y_{mm1}^{(l)}$ is nilpotent for $1\le l\le r-1$.
%We recall the well-known fact that 
{Since} $\dim\mathcal{N}(\mathfrak{gl}_{a_m})=a_m^2-a_m$,
%Thus
we have the stronger inequality: $$\dim C'_{r-1}({\mathfrak z}_\g(x))-\dim C'_{r-1}({\mathfrak z}_{\g'}(x'))\le (r-1)(\dim{\mathfrak z}_\g(x)-\dim{\mathfrak z}_{\g'}(x'))-(r-1)a_m.$$
(Although we %didn't
used this fact in the proof only in the case $n'=0$, it will be needed to establish the stronger inequality in our main Theorem \ref{precise_dim_commuting_nilpotents}.)
Note that if the characteristic of the ground field is zero or greater than $m$ then in fact $y_{mm1}^{(l)}=0$ by Lemma \ref{redcentlem}.
\end{remark}

\begin{corollary}\label{rough_dim_commuting_matrices}
For all $n\ge 3$ and $r\ge 5$, we have $\dim C_r(\mathfrak{gl}_n)\le (r+1)\frac{n^2-1}{3}+r$.
\end{corollary}

\begin{proof}
For the duration of this proof, denote ${\mathcal N}(\mathfrak{gl}_a)$ by ${\mathcal N}_a$.
Let $r\geq 5$, and let $\mathcal{C}$ be an irreducible component of $C_r(\mathfrak{gl}_n)$.
For each partition $\lambda=[\lambda_1\leq\ldots\leq\lambda_m]$ of $n$ let $X_\lambda$ denote the set of elements of $\mathfrak{gl}_n$ having $m$ distinct eigenvalues of multiplicities $\lambda_1,\ldots ,\lambda_m$.
The subset $X_\lambda$ is Zariski constructible and $\GL_n$-stable, since it is the image of a morphism $\GL_n\times Y\rightarrow{\mathfrak{gl}}_n$, where $Y$ is the open subset of $m$-tuples $((a_1,x_1),\ldots ,(a_m,x_m))\in\prod_{i=1}^m (k\times{\mathcal N}_{\lambda_i})$ such that the $a_i$ are distinct.
As there are finitely many of the subsets $X_\lambda$, and since ${\mathcal C}$ is stable under the dot action of $\GL_r$ (see the proof of Lemma \ref{y_iz_j=y_jz_i}), it follows that for some $\lambda$, $X_\lambda^r$ contains a non-empty open subset ${\mathcal U}$ of ${\mathcal C}$.
Here $m$ is the maximum number of distinct eigenvalues of an element of an $r$-tuple belonging to ${\mathcal C}$.
By the dot action, this is also the maximum number of distinct eigenvalues of any linear combination of the elements of an $r$-tuple belonging to ${\mathcal C}$.
Thus if $(x_1,\ldots ,x_r)\in{\mathcal U}$ then the eigenspace decomposition of $k^n$ for the semisimple part of $x_1$ must be the same as for the semisimple parts of $x_2,\ldots, x_r$.
We therefore have \[ \mathcal{C}=\overline{\mathcal{U}}\subseteq \overline{GL_n\cdot (C_r(\mathcal{N}_{\lambda_1}+kI_{\lambda_1})\times \cdots \times C_r(\mathcal{N}_{\lambda_m}+kI_{\lambda_m}))}.\] 
Thus
\begin{equation}\label{ineq}
\dim \mathcal{C} \le n^2-\dim N_{\GL_n}(\mathfrak{gl}_{\lambda_1}\times \cdots \times \mathfrak{gl}_{\lambda_m})+\sum _{i=1}^m\left(\dim C_r(\mathcal{N}_{\lambda_i})+r\right).\end{equation}

{By Lemma \ref{rough_dim_commuting_nilpotents} we therefore have to show that
$$n^2-\sum_{i=1}^m\lambda_i^2+\sum_{i=1}^m(r+1)\frac{\lambda_i^2-1}{3}+(m-1)r\le (r+1)\frac{n^2-1}{3}.$$
Rearranging and using $n=\sum_{i=1}^m\lambda_i$ we get an equivalent inequality
\begin{equation}\label{inequality_to_show}
\frac{2r-1}{3}(m-1)\le \frac{2(r-2)}{3}\sum_{1\le i<j\le m}\lambda_i\lambda_j.
\end{equation}
The right hand-side of the above inequality is at least $\frac{r-2}{3}m(m-1)$, so it suffices to show $(2r-1)(m-1)\le (r-2)m(m-1)$. If $m\ge 3$ and $r\ge 5$, then $(r-2)m-(2r-1)=(m-2)r+1-2m\ge 5(m-2)+1-2m=3(m-3)\ge 0$. The inequality (\ref{inequality_to_show}) therefore holds for $r\ge 5$ if $m=1$ or $m\ge 3$. On the other hand, if $m=2$, then it is equivalent to $2r-1\le 2(r-2)\lambda_1\lambda_2$, which holds for $r\ge 4$, as $n\ge 3$ and consequently $\lambda_1\lambda_2\ge 2$.}

\end{proof}

\begin{remark}\label{glbrem}
For the purposes of the proof of Thm. \ref{precise_dim_commuting_nilpotents} we also need to use the fact \cite{Gu:1992} that $\dim C_r(\mathfrak{gl}_n)=0$, resp. $r$, $2(r+1)$, $3(r+2)$ for $n = 0$, resp. $1$, $2$, $3$.
\end{remark}

In order to establish our main theorem on the dimension of $C_r({\mathcal N}(\mathfrak{gl}_n))$, we consider in more detail the case of an element with all Jordan blocks of size $\leq 4$.

\begin{lemma}\label{GLblocksize4}
Let $r\geq 7$ and suppose $x\in{\mathcal N}(\mathfrak{gl}_n)$ has associated partition $[4^a,3^b,2^c,1^d]$ where at least one of $a,b$ is non-zero.
Then $\dim C'(x) < (r+1)\lfloor\frac{n^2}{4}\rfloor$ unless $(a,b,c,d)=(0,1,0,0)$.
In the case $(a,b,c,d)=(0,1,0,0)$ we have the weaker inequality $\dim C'(x)\leq (r+1)\frac{n^2}{4}$.
\end{lemma}

\begin{proof}
We wish to show
\begin{align}\label{need to show}
\dim C'_{r-1}(\z_\g(x)) < (r+1)\left\lfloor \frac{n^2}{4}\right\rfloor-n^2+\dim\z_\g(x)
\end{align}
with the exception of the case $(a,b,c,d)=(0,1,0,0)$.
We may assume $x$ is written in the form indicated in the proof of Lemma \ref{redcentlem}.
Consider a tuple $(y_1,\ldots ,y_{r-1})$ in $C'_{r-1}(\z_\g(x))$. Then for $1\le i\le r-1$, each $y_i$ is of the form
\begin{equation}\label{glmat}
\left[
\begin{array}{cccccccccc}A_1^{(i)}&A_2^{(i)}&A_3^{(i)}&A_4^{(i)}&B_1^{(i)}&B_2^{(i)}&B_3^{(i)}&C_1^{(i)}&C_2^{(i)}&D^{(i)}\\0&A_1^{(i)}&A_2^{(i)}&A_3^{(i)}&0&B_1^{(i)}&B_2^{(i)}&0&C_1^{(i)}&0\\0&0&A_1^{(i)}&A_2^{(i)}&0&0&B_1^{(i)}&0&0&0\\0&0&0&A_1^{(i)}&0&0&0&0&0&0\\0&E_1^{(i)}&E_2^{(i)}&E_3^{(i)}&F_1^{(i)}&F_2^{(i)}&F_3^{(i)}&G_1^{(i)}&G_2^{(i)}&H^{(i)}\\0&0&E_1^{(i)}&E_2^{(i)}&0&F_1^{(i)}&F_2^{(i)}&0&G_1^{(i)}&0\\0&0&0&E_1^{(i)}&0&0&F_1^{(i)}&0&0&0\\0&0&J_1^{(i)}&J_2^{(i)}&0&K_1^{(i)}&K_2^{(i)}&L_1^{(i)}&L_2^{(i)}&M^{(i)}\\0&0&0&J_1^{(i)}&0&0&K_1^{(i)}&0&L_1^{(i)}&0\\0&0&0&N^{(i)}&0&0&P^{(i)}&0&Q^{(i)}&R ^{(i)}
\end{array}
\right].
\end{equation}
By assumption the characteristic is either zero or greater than 4, hence by Lemma \ref{redcentlem} we have $A_1^{(i)}=0$, $F_1^{(i)}=0$, $L_1^{(i)}=0$ and $R^{(i)}=0$.

From the rank condition for $C'_{r-1}(\z_\g(x))$, we need 
\[ (x+\lambda y_i+\mu y_j)^4=0\]
for all $\lambda ,\mu \in k$ and $1\le i,j\le r-1$. Since $\chr k\not \in\{2,3\}$, we have $x^2y_iy_j=0$ for $1\le i,j\le r-1$, which implies that $B_1^{(i)}E_1^{(j)}=0$.
Hence $(B_1^{(1)},\ldots ,B_1^{(r-1)},E_1^{(1)},\ldots ,E_1^{(r-1)})\in{\mathcal Y}_{r-1,a,b}$ (see Lemma \ref{y_iz_j=0}).
%Let
%\[\mathcal{V}=\{(z_1,..,z_{r-1},w_1,..,w_{r-1})\in M_{a\times b}^{r-1}\times M_{b\times a}^{r-1}: z_iw_j=0, 1\le i, j\le r-1\}.\]
In what follows we consider various projections, applying the theorem on dimensions of fibres of morphisms to produce the desired upper bound on the dimension.
First we let 
\[ \pi _1\colon C'_{r-1}(\mathfrak{z}_\g(x))\to \mathcal{Y}_{r-1,a,b}\times \mathfrak{gl}_a^{3(r-1)}\times M_{a\times b}^{2(r-1)}\times M_{b\times a}^{2(r-1)}\times M_{a\times c}^{r-1}\times M_{c\times a}^{r-1}\]
be the projection mapping each tuple $(y_1,\ldots ,y_{r-1})$ to a collection of tuples of $B_1$'s, $E_1$'s, $A_2$'s, $A_3$'s, $A_4$'s, $B_2$'s, $B_3$'s, $E_2$'s, $E_3$'s, $C_2$'s and $J_2$'s.
Clearly $C'_{r-1}(\mathfrak{z}_\g(x))$ is a conical subset of ${\mathfrak z}_\g(x)^{r-1}$, hence each irreducible component contains $(0,\ldots ,0)$.
The theorem on dimensions of fibres therefore implies
\begin{align}\label{inequality 1}
\dim C'_{r-1}(\mathfrak{z}_\g(x)) 
\le (r-1)a(3a+4b+2c)+\dim {\mathcal Y}_{r-1,a,b} +\dim \pi _1^{-1}(0,\ldots ,0).
\end{align}

We observe now that a tuple $(y_1,\ldots ,y_{r-1})\in \pi _1^{-1}(0,\ldots ,0)$ satisfies $D^{(i)}N^{(j)}=D^{(j)}N^{(i)}$ for all $i,j\le r-1$.
Let ${\mathcal W}_{r-1,a,d}$ be the variety defined in Lemma \ref{y_iz_j=y_jz_i}, and let $\pi_2\colon \pi^{-1}(0,\ldots ,0)\to{\mathcal W}_{r-1,a,d}$ be the projection sending $(y_1,\ldots ,y_{r-1})$ to the collection of $D$'s and $N$'s.
%Let
%$$\mathcal{W}=\{(z_1,\ldots ,z_{r-1},w_1,\ldots ,w_{r-1})\in M_{a\times d}^{r-1}\times M_{d\times a}^{r-1}:z_iw_j=z_jw_i,1\le i,j\le r-1\}$$
%and let $\pi _2\colon \pi ^{-1}(0,\ldots ,0)\to \mathcal{W}$ be the projection sending $(y_1,\ldots ,y_{r-1})$ to the collection of $D$'s and $N$'s.
Then we have 
\begin{align}\label{inequality 2}
\dim \pi_1^{-1}(0,\ldots ,0)\leq \dim \pi_2^{-1}(0,\ldots ,0)+\dim {\mathcal W}_{r-1,a,d}.
\end{align}
% theorem on dimensions of fibres together with Lemma \ref{y_iz_j=y_jz_i} imply
%\begin{align}\label{inequality 2}
%\dim \pi _1^{-1}(0,\ldots ,0)\le \dim \pi _2^{-1}(0,\ldots ,0)+\min \{rad+\lfloor \frac{d^2}{2}\rfloor ,2(r-1)ad\}.
%\end{align}

Now suppose $(y_1,\ldots ,y_{r-1})\in \pi _2^{-1}(0,\ldots ,0)$.
Let 
\[\pi _3\colon \pi _2^{-1}(0,\ldots ,0)\to \mathfrak{gl}_b^{r-1}\times M_{b\times c}^{r-1}\times M_{c\times b}^{r-1}\times M_{b\times d}^{r-1}\times M_{d\times b}^{r-1}\]
be the projection sending $(y_1,\ldots ,y_{r-1})$ to a collection of tuples of $F_3$'s, $G_2$'s, $K_2$'s, $H$'s, and $P$'s.
Once more, the theorem on dimensions of fibres gives us
\begin{align}\label{inequality 3}
\dim \pi _2^{-1}(0,\ldots ,0)\le (r-1)b(b+2c+2d)+\dim \pi _3^{-1}(0,\ldots ,0).
\end{align}

Next, consider $(y_1,\ldots ,y_{r-1})\in \pi _3^{-1}(0,\ldots,0)$. Then $[y_i,y_j]=0$ implying $[F_2^{(i)},F_2^{(j)}]=0$ for all $1\le i,j\le r-1$.
By considering the map 
$$\pi _4\colon \pi _3^{-1}(0,\ldots ,0) \to C_{r-1}(\mathfrak{gl}_b), \quad
(y_1,\ldots ,y_{r-1})\mapsto(F_2^{(1)},\ldots ,F_2^{(r-1)}),
$$
we obtain
\begin{align}\label{inequality 4}
\dim \pi _3^{-1}(0,\ldots ,0)\le \dim \pi _4^{-1}(0,\ldots ,0)+\dim C_{r-1}(\mathfrak{gl}_b).
\end{align}
Suppose $(y_1,\ldots ,y_{r-1})\in \pi _4^{-1}(0,\ldots ,0)$.
Then the commutativity relations imply that for all $1\le i,j\le r-1$ we have
\begin{equation}\label{Wconds}
\begin{array}{ccc}
C_1^{(i)}L_2^{(j)}=C_1^{(j)}L_2^{(i)} ,& G_1^{(i)}L_2^{(j)}=G_1^{(j)}L_2^{(i)}, & L_2^{(i)}J_1^{(j)}=L_2^{(j)}J_1^{(i)},\\
L_2^{(i)}K_1^{(j)}=L_2^{(j)}K_1^{(i)}, & C_1^{(i)}J_1^{(j)}=C_1^{(j)}J_1^{(i)}, & G_1^{(i)}K_1^{(j)}=G_1^{(j)}K_1^{(i)}.
\end{array}
\end{equation}
Moreover, the rank condition $(y_1,\ldots ,y_{r-1})\in C'_{r-1}({\mathfrak z}_\g(x))$ implies that $G_1^{(i)}K_1^{(j)}=0$. (Note that the constant term of the fourth block in the first row of the matrix $(x+\lambda y_i+\mu y_j)^3$ is $I_a$, the quadratic term of the seventh block in the fifth row of the same matrix is $3(\lambda G_1^{(i)}+\mu G_1^{(j)})(\lambda K_1^{(i)}+\mu K_1^{(j)})$, there are no linear terms in that matrix, and that $\chr k\not \in \{2,3\}$). 
In other words, the set of tuples of $C_1$'s, ${G}_1$'s, ${J}_1$'s, $K_1$'s, and $L_2$'s belongs to the set ${\mathcal Z}_{r-1,a,b,c}$ defined in Lemma \ref{y,z,w,v,u}.
%satisfying the conditions \eqref{Wconds} and $G_1^{(i)}K_1^{(j)}=0$ for all $1\leq i,j\leq r-1$.
%Then ${\mathcal Z}$ is isomorphic to the set defined in Lemma \ref{y,z,w,v,u} (with $r$ replaced by $(r-1)$).
Let $\pi _5$ be the projection from $\pi _4^{-1}(0,\ldots ,0)$ to $\mathcal{Z}_{r-1,a,b,c}$.
Then
\begin{align}\label{inequality 5}
\dim \pi _4^{-1}(0,\ldots ,0)\le \dim \mathcal{Z}_{r-1,a,b,c}+\dim \pi _5^{-1}(0,\ldots ,0).
\end{align}   
On the other hand, $\pi _5^{-1}(0,\ldots ,0)$ is isomorphic to the variety of all $(r-1)$-tuples
$$\left(\left[
\begin{array}{ccc} 0 &0&M^{(1)}\\0& 0 &0\\0&Q^{(1)}& 0
\end{array}
\right],\ldots ,\left[
\begin{array}{ccc} 0 &0&M^{(r-1)}\\0&0&0\\0&Q^{(r-1)}&0
\end{array}
\right]\right)$$
belonging to $C'_{r-1}\left(\mathfrak{z}_{\mathfrak{gl}_{2c+d}}\left(\left[
\begin{array}{ccc}0&I_c&0\\0&0&0\\0&0&0
\end{array}
\right]\right)\right)$.
The rank conditions imply $ M^{(i)}Q^{(i)}=0$  for each $1\le i\le r-1$ and that the determinant of every $(c+1)\times (c+1)$ submatrix of $\begin{bmatrix} \xi I_c & M^{(i)} \\ Q^{(i)} & 0 \end{bmatrix}$ is zero. Considering the possible coefficients of $\xi ^{c-1}$ we obtain $Q^{(i)}M^{(i)}=0$ for each $1\le i\le r-1$ too.
Lemma \ref{yz=0,zy=0} then implies 
\begin{align}\label{inequality 6}
\dim \pi _5^{-1}(0,\ldots ,0){\le}(r-1)cd.
\end{align} 
Combining \eqref{inequality 1}-\eqref{inequality 6}, we have
\begin{align*}
\dim C'_{r-1}(\z_\g(x))&\le (r-1)(a(3a+4b+2c)+b(b+2c+2d)+cd)+\dim{\mathcal Y}_{r-1,a,b}\\&+\dim {\mathcal W}_{r-1,a,d}+\dim C_{r-1}(\mathfrak{gl}_b)+\dim\mathcal Z_{r-1,a,b,c}.
\end{align*}

We therefore (using Cor. \ref{rough_dim_commuting_matrices} and \cite{Gu:1992} (see Rk. \ref{glbrem}) for $C_{r-1}(\mathfrak{gl}_b)$)  obtain various upper bounds for $\dim C'_{r-1}(\z_\g(x))$ for various ranges of values of $a,b,c,d$.
We can then verify computationally that in almost all cases the right-hand side is less than the desired amount in \eqref{need to show}.
The exceptions can be dealt with by a more direct and precise determination of $\dim C'(x)$.
The details are in Appendix \ref{GLApp}.
\end{proof}

This completes the preparation we require for our main theorem.

\begin{theorem}\label{precise_dim_commuting_nilpotents}
Assume ${\rm char}\, k\neq 2,3$.
For each $n\ge 4$ and $r\ge 7$, we have 
\[ \dim C_r(\mathcal{N}(\mathfrak{gl}_n))=(r+1)\left\lfloor \frac{n^2}{4}\right\rfloor \]
and the irreducible component(s) of maximal dimension is, resp. are: $G\cdot {\mathfrak u}_{m,m}^r$ if $n=2m$, resp. $G\cdot{\mathfrak u}_{m,m+1}^r$, $G\cdot{\mathfrak u}_{m+1,m}^r$ if $n=2m+1$.
\end{theorem}

\begin{proof}
By Cor. \ref{square_zeroC'(x)}, we have $\dim C_r({\mathcal N}(\mathfrak{gl}_n))\geq (r+1)\lfloor \frac{n^2}{4}\rfloor$.
Suppose $x\in{\mathcal N}(\mathfrak{gl}_n)$ has associated partition $[m^{a_m},\ldots ,1^{a_1}]$ where $a_m>0$.
By Lemma \ref{GLblocksize4} {and Cor. \ref{square_zeroC'(x)}}, it will suffice to show that $\dim C'_{r-1}({\mathfrak z}_\g(x))< (r+1)\lfloor \frac{n^2}{4}\rfloor-n^2+\dim{\mathfrak z}_\g(x)$ whenever $m>4$.
Write $x$ in the form indicated in Lemma \ref{redcentlem}, let $n'=n-ma_m$, let $\g'=\mathfrak{gl}_{n'}$ and let $x'\in\g'$ be the submatrix of $x$ as indicated in the proof of Lemma \ref{rough_dim_commuting_nilpotents}.
Suppose by induction that $\dim C'_{r-1}(\z_{\g'}(x'))\leq (r+1)\frac{{n'}^2}{4}-n'^2+\dim\z_{\g'}(x')$.
Recall from the proof of Lemma \ref{rough_dim_commuting_nilpotents} that $\dim \z_\g(x)-\dim\z_{g'}(x')=a_m(2n-ma_m)$ and $n^2-{n'}^2 = ma_m(2n-ma_m)$.
By Remark \ref{inductioninequality} we have:
\begin{align*}
\dim C'_{r-1}(\z_\g(x)) & \leq \dim C'_{r-1}(\z_{\g'}(x')) + (r-1)(\dim\z_\g(x)-\dim\z_{\g'}(x'))-(r-1)a_m \\
 & \leq (r+1)\frac{{n'}^2}{4}-{n'}^2+\dim\z_{\g'}(x')+(r-1)(\dim\z_\g(x)-\dim\z_{\g'}(x'))-(r-1)a_m \\
 & = (r+1)\frac{n^2}{4}-n^2+\dim\z_\g(x) -a_m(2n-ma_m)\left(\frac{r-3}{4}m-(r-2)\right)-(r-1)a_m.
\end{align*}
Since $\frac{r-3}{4}m-(r-2)\geq \frac{r-7}{4}\geq 0$, it follows that $\dim C'_{r-1}({\mathfrak z}_\g(x))\leq (r+1)\frac{n^2}{4}-n^2+\dim\z_\g(x)-(r-1)a_m$, and therefore $\dim C'(x)\leq  (r+1)\frac{n^2}{4}-(r-1)a_m$.
As $(r+1)\frac{n^2}{4}-(r-1)a_m<(r+1)\frac{n^2-1}{4}\leq (r+1)\lfloor \frac{n^2}{4}\rfloor$, this completes our proof.
\end{proof}

Consequently, we obtain the dimension of $C_r(\mathfrak{gl}_n)$ and $C_r(\mathfrak{sl}_n)$ for the same values of $n$ and $r$ (and with the same assumption on the characteristic).

\begin{corollary}\label{precise_dim_commuting_matrices}
Assume ${\rm char}\, k\neq 2,3$, $n\geq 4$ and $r\geq 7$.
Let ${\mathfrak m}$ be a nil subalgebra of $\mathfrak{gl}_n$ of maximal dimension.
Then $$\dim C_r(\mathfrak{gl}_n) = \left\{ \begin{array}{ll} (r+1)\lfloor \frac{n^2}{4}\rfloor+r & \mbox{if $(n,r)\neq (4,7)$,} \\
40 & \mbox{if $(n,r)=(4,7)$.} \end{array} \right.$$
If $(n,r)\neq (4,7)$ then $G\cdot (kI_n+{\mathfrak m})^r$ is an irreducible component of maximal dimension; if $(n,r)\in \{ (4,7), (5,7), (4,8)\}$ then the generic component {$\overline{\GL_n\cdot \mathfrak{h}^r}$, where $\mathfrak{h}$ is a Cartan subalgebra of $\mathfrak{gl}_n$,} is an irreducible component of maximal dimension.

\end{corollary}

\begin{proof}
Clearly $C_r(\mathcal{N}(\mathfrak{gl}_n))+(kI_n)^r\subseteq C_r(\mathfrak{gl}_n)$, therefore \[
\dim C_r(\mathfrak{gl}_n)\ge \dim C_r(\mathcal{N}(\mathfrak{gl}_n))+r=(r+1)\left\lfloor \frac{n^2}{4}\right\rfloor +r\]
for all $n\ge 4$ and $r\ge 7$.

Conversely, arguing as in Corollary \ref{rough_dim_commuting_matrices}, and using Theorem \ref{precise_dim_commuting_nilpotents} instead of {Lemma} \ref{rough_dim_commuting_nilpotents}, we obtain
\begin{align*}
\dim \mathcal{C} &\le n^2-\dim N_{GL_n}(\mathfrak{gl}_{\lambda_1}\times \cdots \times \mathfrak{gl}_{\lambda_m})+\sum _{i=1}^m\left(\dim C_r(\mathcal{N}(\mathfrak{gl}_{\lambda_i}))+r\right)\\
&\le n^2-\sum _{i=1}^m\lambda_i^2+mr+\sum _{i=1}^l(\lambda^2_i-2\lambda_i+1+r(\lambda_i-1))+\sum _{i=l+1}^m(r+1)\left\lfloor \frac{\lambda_i^2}{4}\right\rfloor
\end{align*}
for any component $\mathcal{C}$ of $C_r(\mathfrak{gl}_n)$, where $\lambda_i\le 3$ for $i\le l$ and $\lambda_i\ge 4$ for $i>l$.
If $l=m$, i.e. $\lambda_i\le 3$ for all $i$,
then all the varieties $C_r(\mathfrak{gl}_{\lambda_i})$ are irreducible by \cite{Gu:1992}, which implies that $\mathcal{C}$ is the generic component (see e.g. the proof of Lemma 5 in \cite{Si}). Hence $\dim \mathcal{C}=n^2+(r-1)n$. It is easy to see that $n^2+(r-1)n\le (r+1)\lfloor \frac{n^2}{4}\rfloor +r$ for all $r\ge 3+\frac{8}{n-3}$ if $n\ge 5$ is odd and for all $r\ge 3+\frac{8}{n-2}+\frac{4}{(n-2)^2}$ if $n\ge 4$ is even. Hence, $\dim \mathcal{C}\le (r+1)\lfloor \frac{n^2}{4}\rfloor +r$ if $r\ge 7$, except if $r=7$ and $n=4$. Moreover, we have the equality $\dim \mathcal{C}=(r+1)\lfloor \frac{n^2}{4}\rfloor +r$ if and only if $(n,r)\in \{(4,8),(5,7)\}$.

Now we suppose that there is some $\lambda_i\ge 4$.
We need to prove the following
\begin{align*}
n^2-\sum _{i=1}^m\lambda_i^2+mr+\sum _{i=1}^l\left(\lambda^2_i-2\lambda_i+1+r(\lambda_i-1)\right)+\sum _{i=l+1}^m(r+1)\left\lfloor \frac{\lambda_i^2}{4}\right\rfloor\le(r+1)\left\lfloor \frac{n^2}{4}\right\rfloor +r.
\end{align*}
Since $n\ge 4$, this inequality is obvious for $m=1$ (in fact, we have equality in this case), therefore we assume $m\ge 2$. Observe that $\lfloor \frac{\lambda_i^2}{4}\rfloor \le \frac{\lambda_i^2}{4}$ for each $i$ and $\lfloor \frac{n^2}{4}\rfloor\ge \frac{n^2-1}{4}$, and that we have shown above that $\lambda_i=1$ for $i\le l$.
So it suffices to show that
$$(r-3)\sum_{i=l+1}^m\frac{\lambda_i^2}{4}+(m-1)r+\frac{r+1}{4}-l<(r-3)\frac{n^2}{4}.$$
Since $n=l+\sum_{i=l+1}^m\lambda_i$,
the last inequality can be deduced from the following
\begin{align}\label{m,l}
 (m-1)r+\frac{r+1}{4}-l<(r-3)\Big(\sum_{l+1\le i<j\le m}\frac{\lambda_i\lambda_j}{2} +\frac{l}{2}\sum _{i=l+1}^m\lambda_i+\frac{l^2}{4}\Big). 
\end{align}

Assume first that $l\ge 1$. Since $\lambda_{i_0}\ge 4$ for some $i_0$, we have
$$\sum_{l+1\le i<j\le m}\frac{\lambda_i\lambda_j}{2}+\frac{l}{2}\sum _{i=l+1}^m\lambda_i\ge 2(m-1).$$
To prove (\ref{m,l}) it therefore suffices to show
$$\frac{r-3}{4}\le (r-6)(m-1)$$
which holds for any $r\ge 7$ and $m\ge 2$.

On the other hand, if $l=0$, then $\lambda_i\ge 4$ for each $i$, therefore
$$\sum _{l+1\le i<j\le m}\frac{\lambda_i\lambda_j}{2}\ge 4m(m-1),$$
which implies that (\ref{m,l}) holds for any $m\ge 2$ and $r\ge 4$.
\end{proof}

\begin{corollary}\label{precise_dim_commuting_matrices_sl}
Keep the assumptions and notation of Cor. \ref{precise_dim_commuting_matrices}.
If ${\rm char}\, k|n$ then $\dim C_r(\mathfrak{sl}_n)=\dim C_r(\mathfrak{gl}_n)$ and $G\cdot (kI_n+{\mathfrak m})^r$ is an irreducible component of maximal dimension.

Otherwise, $\dim C_r(\mathfrak{sl}_n)=(r+1)\lfloor \frac{n^2}{4}\rfloor=\dim C_r({\mathcal N}(\mathfrak{sl}_n))$ unless $(n,r)=(4,7)$, in which case the dimension is 33.
Clearly $G\cdot {\mathfrak m}^r$ is an irreducible component of maximal dimension if $(n,r)\neq (4,7)$; the generic component has maximal dimension for $(n,r)\in \{ (4,7), (4,8), (5,7)\}$.
\end{corollary}

\begin{proof}
We note that if ${\rm char}\, k|n$ then $kI_n+\mathfrak{m}\subseteq \mathfrak{sl}_n$, so ${\rm SL}_n\cdot (kI_n+\mathfrak{m})^r = \GL_n\cdot (kI_n+{\mathfrak m})^r$ is contained in $C_r(\mathfrak{sl}_n)$.
The statement about the dimension in this case follows.
(Since ${\rm char}\, k\neq 2$, we cannot have $n=4$ here.)
For the remaining cases, we argue as in the proof of Cor. \ref{precise_dim_commuting_matrices}.
Note that our assumption on the characteristic implies that some $\lambda _i$ is not divisible by the characteristic. 
Consequently,
$$\dim \mathcal{C}\le n^2-\sum _{i=1}^m\lambda _i^2+\sum _{i=1}^m(\dim C_r(\mathcal{N}(\mathfrak{gl}_{\lambda _i}))+r)-r$$
and the conclusion follows by exactly the same argument as above if we slightly adapt the proofs of \cite{Gu:1992} and \cite[Lemma 5]{Si} to hold also in $\mathfrak{sl}_n$.
\end{proof}

\begin{remark}
a) The inequalities $\dim C_r(\mathcal{N}(\mathfrak{gl}_n))\le (r+1)\lfloor \frac{n^2}{4}\rfloor$ and $\dim C_r(\mathfrak{gl}_n)\le (r+1)\lfloor \frac{n^2}{4}\rfloor+r$ do not hold for small values of $r$. Recall that the regular component of $C_r(\mathcal{N}(\mathfrak{gl}_n))$ has dimension $(n-1)(r+n-1)$ (see \cite[Proposition 1]{NS}) which for $r,n\geq 4$ is greater than $(r+1)\lfloor \frac{n^2}{4}\rfloor$ if and only if  $r=4$ and $n\in \{4,5\}$.
Similarly, the generic component of $C_r(\mathfrak{gl}_n)$ has dimension $n^2+(r-1)n$ which, for $r,n\geq 4$ is greater than $(r+1)\lfloor \frac{n^2}{4}\rfloor +r$ if and only if 
$$(n,r)\in \{(4,4),(4,5),(4,6),(4,7),(5,4),(5,5),(5,6),(6,4),(6,5),(7,4),(8,4),(9,4),(10,4)\}.$$
We conjecture that these are the only values of $n\ge 4$ and $r\ge 4$ such that $\dim C_r(\mathcal{N}(\mathfrak{gl}_n))> (r+1)\lfloor \frac{n^2}{4}\rfloor$ resp. $\dim C_r(\mathfrak{gl}_n)>(r+1)\lfloor \frac{n^2}{4}\rfloor+r$.

b) The proof of of Thm. \ref{precise_dim_commuting_nilpotents} starts with the `base case' of a nilpotent element $x$ for which the Jordan blocks are of order at most four; once the base case is established, an induction argument proves the general case.
Given such a base case, the induction argument can only be applied when $r\geq 7$.
To determine the dimension when ${4}\leq r\leq 6$ with our methods, one would therefore have to enlarge the base case to consider Jordan blocks of larger orders.
\end{remark}

\section{Cohomological complexity theory}\label{complexity}

In this section we will discuss some applications of our earlier results to complexity of modules over finite group schemes, using a powerful result of Suslin-Friedlander-Bendel linking nilpotent commuting varieties with cohomological support varieties.
While the main results of this paper concern the type $A$ only, these applications suggested a more general statement about complexities of rational $G$-modules (Theorem \ref{complexity bound}) which we prove in all types.
The groups we are most interested in are Frobenius kernels and finite Chevalley groups.
We start with some notation and background for group schemes. 

\subsection{Notation and background}
We assume from now on that ${\rm char}\, k=p>0$ and $k=\overline{\mathbb{F}}_p$. Let $G$ be a reductive group scheme defined over $\mathbb F_p$.
(See \cite[I.2]{Jan:2003} for a definition.)
For each positive integer $r$, let $F_r:G\to G^{(r)}$ be the $r$-th Frobenius morphism (see \cite[I.9]{Jan:2003}).
The scheme-theoretic kernel $G_{(r)}=\ker(F_r)$ is called the \emph{$r$-th Frobenius kernel} of $G$.
On the other hand, the fixed-point subgroup $G(\mathbb F_{p^r})=G^{F_r}$ is a finite Chevalley group.
The subgroup schemes $G_{(r)}$ and $G(\mathbb{F}_{p^r})$ are important examples of {\it finite group schemes}, i.e. group schemes with a finite-dimensional coordinate algebra.
Fix $B \subseteq G$ a Borel subgroup of $G$ and let $U \subseteq B$ be the unipotent radical of $B$.
We also have the corresponding notation $B_{(r)}, U_{(r)}, B(\mathbb F_{p^r})$, and $U(\mathbb F_{p^r})$. Recall that $\mathfrak g=\text{Lie}(G)$, and let $\mathfrak u=\text{Lie}(U)$.

By convention, our $G$-modules are assumed to be finite-dimensional and rational over $G$.
The cohomological background below is valid for any finite group scheme $H$, but we will only consider $H$ to be one of $G_{(r)},B_{(r)},U_{(r)},G(\mathbb F_{p^r}),B({\mathbb F}_{p^r}),U({\mathbb F}_{p^r})$.
For any $H$-module $M$, the $m$-th cohomology of $H$ with coefficients in $M$ is defined as
\[ \opH^m(H,M)=\Ext^m_H(k,M).\]
We further denote 
\[ \Hbul(H,M)=\bigoplus_{m=0}^\infty\opH^m(H,M)\] and 
\[ \opH^{ev}(H,k)=
\begin{cases}
\displaystyle{\bigoplus_{m=0}^\infty\opH^{2m}(H,k)}~&\text{if}~p>2,\\
\Hbul(H,k)~&\text{if}~p=2.
\end{cases}\]
This is a commutative $k$-algebra under the cup product.
On the other hand, Yoneda composition provides an $ \opH^{ev}(H,k)$-module structure on $\Ext_H^\bullet(M,M)$. Let $J_M$ be the annihilator in $\opH^{ev}(H,k)$ under this action.
We define ${\mathcal V}_H(M)$, the {\it support variety} of $M$ over $H$ to be the maximal ideal spectrum of the quotient ring $\frac{\opH^{ev}(H,k)}{\sqrt{J_M}}$.
In particular, ${\mathcal V}_H(M)$ is a closed subvariety of ${\mathcal V}_H(k)=\spec\, {\opH^{ev}(H,k)_{\red}}$.

Given a graded vector space $\{V_i\}$, the {\it growth rate} of $\{V_i\}$ is defined to be the least integer $s\ge 0$ satisfying
\[ \lim_{n\to \infty}\frac{\dim V_n}{n^s}=0. \] 
If no such $s$ exists, one says that the growth rate for $\{V_i\}$ is infinite.
The {\it complexity} of $M$, denoted $c_H(M)$, is defined to be the growth rate of $\Ext^{\bullet}_H(M,M)$.
It is well known that $c_H(M)=\dim \mathcal{V}_H(M)$ for each finite dimensional $M$, see \cite{NPV:2002}. In particular, $c_H(k)=\dim\opH^{ev}(H,k)$.
 
\subsection{Complexity of modules over Frobenius kernels}
In the two papers \cite{SFB1:1997} and \cite{SFB2:1997}, Suslin, Friedlander, and Bendel determined support varieties of modules over $G_{(r)}$ in terms of nilpotent commuting varieties.
In particular, they showed that if $G$ has a structure of {\it exponential type} (see \cite[1.6]{Friedlander-rational} for a definition) then
\[ {\mathcal V}_{G_{(r)}}(k)=C_r(\N_{[p]}(\g)),\]
and for each $G$-module $M$ the support variety ${\mathcal V}_{G_{(r)}}(M)$ is a closed, 
conical $G$-subvariety of $C_r(\N_{[p]}(\g))$. We recall that the restricted 
nullcone $\N_{[p]}(\g)$ is a closed subvariety of $\N(\g)$, hence 
$C_r(\N_{[p]}(\g))\subseteq C_r(\N(\g))$.
Our calculation of the dimension of 
$C_r(\N(\g))$ provides insight into the complexity of $G$-modules.

\begin{theorem}\label{complexity of G_r}
Let $G=\GL_n$ and suppose $p>3$. If $n\ge 4, r\ge 7$, then for any $G_{(r)}$-module $M$, we have \[ c_{G_{(r)}}(M)\le 
(r+1)\lfloor \frac{n^2}{4}\rfloor 
\] with equality if the dimension of $M$ is not divisible by $p$.
\end{theorem}
\begin{proof}
As $\GL_n$ is of exponential type, we have from the result of Suslin-Friedlander-Bendel that
\[  c_{G_{(r)}}(M)\le \dim C_r(\N_{[p]}(\g))\le \dim C_r(\N(\g)). \]
%The dimension on the right-hand side is 
%$(r+1)\lfloor\frac{n^2}{4}\rfloor$ for $G=\GL_n$ and $(r+1)\frac{n^2+n}{2}$ 
%for $G=\Sp_{2n}$, which establishes the upper bound.
By \cite[Example 6.9]{SFB2:1997}, we have \[c_{G_{(r)}}(M)=c_{G_{(r)}}(k)=\dim C_r(\N_{[p]}(\g))\]
for every $G$-module $M$ whose dimension is not divisible by $p$.
On the other hand, if ${\mathfrak m}$ is a commutative nil subalgebra of maximal dimension (i.e. ${\mathfrak m}={\mathfrak u}_{m,m}$ if $\g=\mathfrak{gl}_{2m}$, ${\mathfrak m}={\mathfrak u}_{m,m+1}$ or ${\mathfrak u}_{m+1,m}$ if $\g=\mathfrak{gl}_{2m+1}$) then we observe that 
\[ G\cdot{\mathfrak m}^r\subseteq C_r(\N_{[p]}(\g)). \]
Hence, our assertion follows from Theorem \ref{precise_dim_commuting_nilpotents}. 
\end{proof}

\begin{remark}
For the case $p=2$, the second author already proved the equality for $G=\GL_n$ and $n, r\ge 1$, see \cite[Proposition 3.3.4]{Ngo:2014}.
We expect the theorem also to hold for the case $p=3$.
\end{remark}

We now consider what can be said about the complexity of the restriction of an arbitrary rational $G$-module to $G_{(r)}$.
If $M$ has dimension coprime to $p$ then the support variety of $M$ over $G_{(r)}$ is the same as the support variety of the trivial module over $G_{(r)}$. %(\cite[Example 6.9]{SFB2:1997}).
%(This can easily be seen by interpreting this support variety as the set of one-parameter subgroups $\varepsilon:{\mathbb G}_{a(r)}\rightarrow G$ such that $\varepsilon^*M$ is not projective.)
On the other hand, there are in general many other rational $G$-modules $M$ for which $c_{G_{(r)}}(M)=c_{G_{(r)}}(k)$ but ${\mathcal V}_{G_{(r)}}(M)\subsetneq{\mathcal V}_{G_{(r)}}(k)$.

To explain this, assume in the following discussion that $G=\SL_{n}$.
(Restriction from $\GL_{n}$ to $\SL_{n}$ clearly has no effect on support varieties.)
Fix a maximal torus $T$ contained in $B$, and let $\Delta=\{\alpha_1,\ldots ,\alpha_{n-1}\}$ be the corresponding basis of simple roots (numbered in the standard way).
Let $X(T)$, resp. $Y(T)$, be the lattice of characters, resp. cocharacters in $T$.
Since $G$ is simply-connected, $Y(T)=\oplus_{i=1}^n {\mathbb Z}\alpha_i^\vee$.
Let $\omega_1,\ldots ,\omega_{n-1}$ be the fundamental weights, i.e. such that $\langle \omega_j,\alpha_i^\vee\rangle=\delta_{ij}$ for $1\leq i,j\leq n-1$, where $\langle .\, ,.\rangle$ denotes the perfect pairing $X(T)\times Y(T)\rightarrow{\mathbb Z}$.
Then $X(T)=\oplus_{i=1}^{n-1} {\mathbb Z}\omega_i$.
A weight $\lambda\in X(T)$ is \emph{$p^r$-restricted} (dominant) if $0\leq\langle \lambda, \alpha_i^\vee\rangle<p^r$ for $1\leq i\leq n-1$; then $\lambda=\lambda_0+p\lambda_1+\ldots +p^{r-1}\lambda_{r-1}$ where $\lambda_0,\ldots ,\lambda_{r-1}$ are $p$-restricted.
For a dominant weight $\lambda$ let $L(\lambda)$ denote the simple (finite-dimensional) $G$-module with highest weight $\lambda$.
If $\lambda$ is $p^r$-restricted then $L(\lambda)|_{G_{(r)}}$ is also irreducible.
Sobaje proved \cite[Thm. 3.2]{So} that for $p$ large enough, we have:
$${\mathcal V}_{G_{(r)}}(L(\lambda))=\left\{ (x_0,\ldots ,x_{r-1})\in C_r({\mathcal N}_{[p]}(\g)): x_i\in{\mathcal V}_{G_{(1)}}(L(\lambda_i))\right\}$$
and therefore that $\dim{\mathcal V}_{G_{(r)}}(L(\lambda))=\dim C_r({\mathcal N}(\g))$ if and only if ${\mathcal V}_{G_{(1)}}(L(\lambda_i))\supseteq G\cdot{\mathfrak m}=\overline{\mathcal O}$ for $0\leq i\leq r-1$, where ${\mathcal O}$ is the maximal nilpotent orbit for which all parts of the partition have order $\leq 2$.
(Recall that ${\mathfrak m}$ is a commutative nil subalgebra of maximal dimension.) This is equivalent to the fact that each $\calV_{G_{(1)}}(L(\lambda_i))$ contains an element $e\in\calO$. Now combining with the algebraic description of $\calV_{G_{(1)}}(L(\lambda_i))$ in \cite[Example 1.18(1)]{FP}, we have proved the following.
%Although , there are certainly many $p$-restricted weights $\lambda_i$ such that ${\mathcal O}\subseteq{\mathcal V}_{G_{(1)}}(L(\lambda_i))\subsetneq {\mathcal N}_{[p]}$.
%Let $\Lambda=(p-1)(\omega_1+\ldots +\omega_n)$ be the Steinberg weight.
%One obtains:

\begin{theorem}\label{maxcomplex}
Let $G=\SL_{n}$ with $n\geq 3$. Suppose $p>n^3/4$ and $r\ge 7$. Let $\lambda=\lambda_0+p\lambda_1+\ldots+p^{r-1}\lambda_{r-1}$ where $\lambda_i$ is $p$-restricted for $0\le i\le r-1$. Then $c_{G_{(r)}}(L(\lambda))=c_{G_{(r)}}(k)$ if and only if for each $i$ an element e of the maximal square-zero orbit has at least one Jordan block of size $< p$ when acting on $L(\lambda_i)$.
\end{theorem}

It would be nice if a geometric description of ${\mathcal V}_{G_{(1)}}(L(\lambda_i))$ could be obtained.
Unfortunately, it is in general a hard problem. If $p$ is extremely large, then one can determine the support variety ${\mathcal V}_{G_{(1)}}(L(\lambda_i))$ combinatorially using \cite[Thm. 4.1]{dnp} and the fact that Lusztig's character formula holds for all restricted dominant weights. Note also that Lusztig's conjecture has been discussed in detail elsewhere, so we only mention Andersen-Jantzen-Soergel's proof (for an unknown but extremely large bound) \cite{ajs} and Williamson's discovery of various counter-examples \cite{Will} for moderately large $p$, indicating that the required bound on $p$ is at least exponential in the rank.

%For $\Sp_{2n}$, for example, this states that if $\lambda$ is a $p^r$-restricted weight chosen at random, then $L(\lambda)|_{G_{(r)}}$ has the same complexity as the trivial module with probability $(1-\frac{1}{p^n})^r$.

\subsection{Connection to finite Chevalley groups}

In the rest of the paper we discuss some observations concerning support varieties and cohomological complexity for modules over the Chevalley group $G({\mathbb F}_{p^r})$, considered in connection with support varieties for $G_{(r)}$.
The main result here was inspired by our calculations on complexity of $G_{(r)}$-modules in the previous subsection.

Interactions between the three categories of $G$-, $G_{(r)}$-, and $G(\mathbb F_{p^r})$-modules have been studied for over fifty years, see for example \cite[Ch. 10]{Hum}.
The restriction functors from $G$-modules to $G_{(r)}$- or $G({\mathbb F}_{p^r})$-modules allow us to relate the cohomology of $G_{(r)}$ to that of $G(\mathbb F_{p^r})$.
However, very little has been established in the way of a direct connection between the categories of $G_{(r)}$- and $G(\mathbb F_{p^r})$-modules.
A distinguished contribution in this direction is a result of Lin and 
Nakano \cite[Thm. 3.4(b)]{LN} stating that
\[ c_{G(\mathbb{F}_p)}(M)\le\frac{1}{2}c_{G_{(1)}}(M) \]
for any $G$-module $M$.
We are going to generalize this to $G_{(r)}$ and $G(\mathbb F_{p^r})$ for all $r\ge 1$.
This extension was inspired by the observation for $p>3$ that:
\[c_{G(\mathbb{F}_{p^{r}})}(k)=\frac{r}{r+1}c_{G_{(r)}}(k)\]
for  $n\ge 4, r\ge 7$ if $G=GL_n$, as follows easily by comparing Theorem \ref{complexity of G_r} with \cite[Table 1, p. 154]{Hum}.
In subsequent work \cite{LNS} we will show that this also holds for $G={\rm Sp}_{2n}$ and a similar range of values of $n$, $r$.
In any case (and for any $G$), we have
\begin{align}\label{complexity_observation}
c_{G(\mathbb{F}_{p^{r}})}(k)\le \frac{r}{r+1}c_{G_{(r)}}(k)
\end{align}
at least for large enough $r$.

In the rest of this paper, we proceed to prove the inequality \eqref{complexity_observation} for an arbitrary simple algebraic group $G$ and $G$-module $M$.
The strategy is to first establish an inequality of the complexities over the Borel subgroups, and then induce up to $G$.
The induction stage requires us to analyze some special geometric properties of $G$-saturation varieties. 

For the next two results, we shall assume that $G$ is a simple algebraic group and $p$ is good for $G$.

\begin{theorem}
Let $W$ be a $B$-stable closed subvariety of $\fraku^r$ for some positive integer $r$. Then we have
\[  \dim G\cdot W\ge \frac{r+1}{r}\dim W. \]
\end{theorem}

\begin{proof}
We can clearly assume that $W$ is irreducible.
Let $\pi_1$ be the restriction to $W$ of the projection from ${\mathfrak u}^r$ onto the first factor, and let $V_1=\overline{\pi_1(W)}$.
We may also assume that $\dim V_1\geq \frac{1}{r}\dim W$.
By the theorem on the dimensions of fibres, we have $\dim G\cdot W=\dim G+\dim W-\min_{w\in W}\dim \{ g\in G : g^{-1}\cdot w\in W\}$.
The inequality $\dim G\cdot W\geq \frac{r+1}{r}\dim W$ is therefore equivalent to: $$\dim G-\min_{w\in W}\dim\{ g\in G: g^{-1}\cdot w\in W\}\geq \frac{1}{r}\dim W.$$

We first consider the case $r=1$, i.e. $W=V_1$.
Since $G\cdot V_1$ is closed and $G$-stable in $\N(\g)$, we can write $G\cdot V_1=\bigcup_{j=1}^m\overline{\calO_{j}}$ where each $\calO_j$ is a nilpotent orbit in $\N(\g)$. Then 
\[ V_1\subseteq  \left(\bigcup_{j=1}^m\overline{\calO_{j}}\right)\cap\fraku=\bigcup_{j=1}^m\left(\overline{\calO_{j}}\cap\fraku\right).\]
By \cite[Theorem 10.11(1)]{Jan:2004}, for each $j$, $\dim\left(\overline{\calO_{j}}\cap\fraku\right)=\frac{1}{2}\dim\calO_j$ and therefore
$2\dim V_1\le \dim G\cdot V_1$.
Hence we have $\dim G-\min_{v\in V_1}\dim\{ g\in G: g^{-1}\cdot v\in V_1\}\geq \dim V_1$.

In the general case, let $\mathcal U_1$ be the open subset of $V_1$ such that $\dim\{ g\in G: g^{-1}\cdot u_1\in V_1\}$ is minimal for all $u_1\in \mathcal U_1$, and let $\mathcal U$ be the corresponding open subset of $W$.
Since $W$, $V_1$ are irreducible and $\pi_1:W\rightarrow V_1$ is dominant, it follows that there exists $u\in\mathcal U$ such that $u_1=\pi_1(u)\in\mathcal U_1$.
Then we clearly have 
\begin{align*}
\dim G-\min_{w\in W}\dim\{ g\in G: g^{-1}\cdot w\in W\} & = \dim G-\dim\{ g\in G: g^{-1}\cdot u\in W\} \\
 & \geq \dim G-\dim\{ g\in G: g^{-1}\cdot u_1\in V_1\}\\
& \geq \dim V_1\geq \frac{1}{r}\dim W.
\end{align*}
\end{proof}

\begin{corollary}\label{inequality of complexity B_r and G_r}
Let $M$ be a $G$-module. For all $r\ge 1$, we have  \[ \displaystyle{\dim \mathcal{V}_{G_{(r)}}(M)\ge\frac{r+1}{r}\dim \mathcal{V}_{B_{(r)}}(M)}.\]
In other words, \[ c_{G_{(r)}}(M)\ge\frac{r+1}{r}c_{B_{(r)}}(M). \]
\end{corollary}

\begin{proof}
We first show that $\mathcal V_{B_{(r)}}(M)$ can be identified with a subvariety of $\fraku^r$. If $G$ is of classical type, then this is done by Lemmas 1.7 and 1.8 in \cite{SFB1:1997}. For other types of $G$, we can make use of Sobaje's results in \cite{So15}\cite{So15b}.
Explicitly, consider the exponential map exp$:\N_{[p]}(\g)\to\mathcal U_1(G)$ defined in \cite[Theorem 3.4]{So15}.
From \cite[Corollary 4.3]{So15b}, this map is in fact the restriction of a Springer isomorphism $\phi:\N(\g)\xrightarrow{\sim}\mathcal U(G)$.
Then it is well-known that $\phi$ restricts to the unipotent radical of any parabolic subgroup of $G$, see for example \cite[Theorem 1.1]{So15b}.
In other words, the map exp restricts to $\N_{[p]}(\fraku)\xrightarrow{\sim} \mathcal U_1(U)$ giving that $U$ has an exponential map defined in \cite[Definition 2.3]{So15} (see \cite[Theorem 3.5]{So15}).
Hence, Theorem 2.5 of \cite{So15} implies an identification between $\mathcal{V}_{U_{(r)}}(k)$ and $C_r(\N_{[p]}(\fraku))$. Since $\mathcal{V}_{B_{(r)}}(k)=\mathcal{V}_{U_{(r)}}(k)$ (see for example the last line of page 23 \cite{NPV:2002}), $\mathcal V_{B_r}(M)$ can be considered as a $B$-stable subvariety of $\fraku^r$.
We note that those results of Sobaje depend on various types of prime $p$ including {\it good}, {\it pretty good}, and {\it separably good}. These types of prime are only different when $G$ is of type $A$, so it does not affect our proof. Finally, the corollary follows immediately from the above theorem and the fact that $G\cdot \mathcal{V}_{B_{(r)}}(M)=\mathcal{V}_{G_{(r)}}(M)$ \cite[Proposition 4.5.2]{Ben:1996}. 
\end{proof}

%Next, we adapt a concept in \cite{SFB1:1997}.
%For each group scheme $\mathcal G$ we denote by $\mathfrak{V}_r(\mathcal G)$ the affine scheme of homomorphisms of group schemes ${\mathbb G}_{a(r)}\rightarrow \mathcal G$, where ${\mathbb G}_a$ denotes the additive group scheme of rank 1.
%Then $\mathfrak{V}_r$ defines a functor from the category of affine group schemes to the category of affine schemes in the obvious way.
%It was proved in \cite{SFB2:1997} that there is a homeomorphism between $\mathfrak V_r(\mathcal G)$ and $\mathcal{V}_{\mathcal G_{(r)}}(k)$. Moreover, for each $G$-module $M$, the support scheme $\mathfrak V_r(\mathcal G)_M$, defined in \cite[Proposition 6.1]{SFB2:1997}, is homeomorphic to $\mathcal{V}_{\mathcal G_{(r)}}(M)$ under the aforementioned homeomorphism. In other words, we have
%\[  c_{\mathcal G_{(r)}}(M)=\dim\mathfrak V_r(\mathcal G)_M,\quad c_{\mathcal B_{(r)}}(M)=\dim\mathfrak V_r(\mathcal B)_M \]
%where $\mathcal B$ is a Borel subgroup scheme of $\mathcal G$.

We first prove our main result in the simplest case.

\begin{lemma}\label{complexity G_a}
For positive integers $r$ and $m$, \[c_{\mathbb G^m_a(\mathbb F_{p^r})}(M)\le c_{(\mathbb G^m_a)_{(r)}}(M)\] for any $\mathbb G^m_a$-module $M$.
\end{lemma}
\begin{proof}
As $\mathbb G^m_a(\mathbb F_{p^r})$ is an elementary abelian group, we have \[\mathbb G^m_a(\mathbb F_{p^r})\cong(\mathbb Z_p)^{mr}\cong \mathbb G^{mr}_a(\mathbb F_{p}).\]
Then applying \cite[Theorem 3.4(3.4.2)]{LN}, we obtain
\[ c_{\mathbb G^{mr}_a(\mathbb F_{p})}(M)\le c_{\mathbb G^{mr}_{a(1)}}(M)=c_{(\mathbb G^m_{a})_{(r)}}(M), \]
where the last equality follows from the equivalence between categories of $\mathbb G_{a(r)}$-modules and $\mathbb G^r_{a(1)}$-modules, \cite[Proof of Prop. 6.5]{SFB2:1997}. Hence, we have proved the inequality.
\end{proof}

Here comes the main theorem.

\begin{theorem}\label{complexity bound}
Let $G$ be a simple algebraic group.
Suppose $p$ is a good prime for $G$.
For any rational $G$-module $M$, we have
\[  c_{U(\mathbb F_{p^r})}(M)\le c_{U_{(r)}}(M), \]
and \[ c_{G(\mathbb F_{p^r})}(M)\le \frac{r}{r+1}c_{G_{(r)}}(M). \]
\end{theorem}

\begin{proof}
Since $U(\mathbb F_{p^r})$ is a $p$-Sylow subgroup of $B(\mathbb F_{p^r})$ and $G(\mathbb F_{p^r})$, 
\[ c_{G(\mathbb F_{p^r})}(M)=c_{B(\mathbb F_{p^r})}(M)=c_{U(\mathbb F_{p^r})}(M). \]
Note further that $c_{U_{(r)}}(M)=c_{B_{(r)}}(M)$, see for example \cite[\S 2.3]{NPV:2002}. 
Hence, it suffices to prove the first inequality as the other one follows immediately from Corollary \ref{inequality of complexity B_r and G_r}.
Observe for any finite group $K$ that $c_{K}(M)$ equals to the maximum complexity of $M$ over all maximal elementary abelian subgroups of $K$, so we just need to prove the inequality for each maximal elementary subgroup of $U(\mathbb F_{p^r})$.
Let $E$ be such a subgroup of $U(\mathbb{F}_{p^r})$.
Then $E\cong\mathbb Z_p^m$ for some positive integer $m$ and is generated by $x_1,\ldots,x_m$.
Using \cite[Theorems 1.3 or 1.4]{Se}, for each $x_i$, we can find a monomorphism from $\mathbb G_a$ to $G$ which sends $1$ to $x_i$.
In other words, there is an abelian unipotent $U_i$ in $G$ which is isomorphic to $\mathbb G_a$ and $U_i(\mathbb F_{p^r})$ contains $x_i$ for each $i$. Now let $U_{max}=\bigoplus_{i=1}^m U_i$, we have $U_{max}\cong\mathbb G_a^m$ and $U_{max}(\mathbb F_{p^r})=E$. Now Lemma \ref{complexity G_a} gives us
\[ c_E(M)=c_{U_{max}(\mathbb F_{p^r})}(M)\le c_{U_{max(r)}}(M).  \]
Since $U_{max(r)}$ is a closed subgroup of $U_{(r)}$, we have \[ \mathcal V_{U_{max(r)}}(M)=\mathcal V_{U_{max(r)}}(k)~\bigcap~ \mathcal V_{U_{(r)}}(M)\] by \cite[2.2.10]{NPV:2002}. It follows that $c_{U_{max(r)}}(M)\le c_{U_{(r)}}(M)$, therefore completing our proof.
\end{proof}

From the fact that the complexity of a module is zero if and only if the module is projective, we obtain an alternative proof to the main result of Drupieski in \cite[Theorem 2.3]{Dru} as follows.

\begin{corollary}\label{finalcor}
Let $M$ be a rational $G$-module. If $M$ is projective over $G_{(r)}$, then it is also projective as a module over $G(\mathbb F_{p^r})$. 
\end{corollary}

Our last theorem also indicates that commuting varieties may be a source of information about complexity over finite Chevalley groups, in addition to the known relationship with support varieties over Frobenius kernels.
This suggests the following question, a positive answer to which would provide further context for Corollary \ref{finalcor}.

\vspace{0.2cm}
\noindent
{\bf Question:} Given a rational $G$-module $M$, is there any connection between the support varieties for $M$ over $G({\mathbb F}_{p^r})$ and over $G_{(r)}$ (or perhaps $G_{(r-1)}$)?

\begin{appendices}

\section{Appendix: Details of computations}\label{GLApp}

We here collect some details of the computations required to establish the desired inequality in the proof of Lemma \ref{GLblocksize4}.
We first verified the required inequality by ad hoc arguments, but it became apparent that a more systematic computational approach would be beneficial.
We used GAP for all of our computations.
Recall that $x$ is an element of ${\mathcal N}(\g)$ of partition type $[4^a,3^b,2^c,1^d]$ and $(a,b)\neq (0,0)$.
First of all, let $$F(b,r)=\left\{ \begin{array}{cc} \left(\frac{r}{3}+\frac{1}{4}\right)b^2+\frac{2r}{3}-1 & \mbox{if $b\geq 3$,} \\
b(r+b-2)+\lfloor \frac{b^2}{4}\rfloor & \mbox{if $b\leq 2$.} \end{array}\right.$$
Then $\dim C_{r-1}(\mathfrak{gl}_b)+\lfloor\frac{b^2}{4}\rfloor\leq F(b,r)$ by Cor. \ref{rough_dim_commuting_matrices} and \cite{Gu:1992}.
Similarly, let $$G(a,b,c,r) = \left\{ \begin{array}{cc} 0 & \mbox{if $c=0$,} \\
(r-1)(a+b)+1 & \mbox{if $c=1$ and $(a,b)\neq (1,0)$,} \\
(r+1) & \mbox{if $(a,b,c)=(1,0,1)$,} \\
\frac{a^2}{11}+c^2+2ac+2bc+(r-2)(c^2+\frac{(a+b)^2}{2}) & \mbox{if $c\geq {\rm max}\, \{a+b,2\}$,} \\
\frac{a^2}{11}+c^2+2ac+2bc+(r-2)(\frac{c^2}{2}+(a+b)c) & \mbox{if $2\leq c\leq a+b$.} \end{array}\right.$$
Let ${\mathcal Z}{=\mathcal{Z}_{r-1,a,b,c}}$ be the variety identified in the proof of Lemma \ref{GLblocksize4}; then $\dim{\mathcal Z}\leq G(a,b,c,r)$ by Lemma \ref{y,z,w,v,u}.
Let $n=4a+3b+2c+d$ and let $$H(a,b,c,d)=n^2-\dim{\mathfrak z}_\g(x)=n^2-(a+b+c+d)^2-(a+b+c)^2-(a+b)^2-a^2,$$ the dimension of the orbit of $x$.
Finally, let
\begin{align*}
N_1(a,b,c,d,r)=& (r+1)\frac{n^2-1}{4}-  H(a,b,c,d) -F(b,r)-G(a,b,c,r) \\ 
 & -(r-1)(3a^2+5ab+2ac+2ad+b^2+2bc+2bd+cd),\\
N_2(a,b,c,d,r)=& (r+1)\frac{n^2-1}{4}-  H(a,b,c,d) -F(b,r)-G(a,b,c,r) \\ 
 & -(r-1)(3a^2+5ab+2ac+b^2+2bc+2bd+cd)-\min\{rad+\lfloor\frac{d^2}{2}\rfloor,2(r-1)ad\}.
\end{align*}
The inequality $N_2(a,b,c,d,r) > 0$ is slightly stronger than the desired inequality in Lemma \ref{GLblocksize4}, since we have replaced $\lfloor \frac{n^2}{4}\rfloor$ by $\frac{n^2-1}{4}$ and $\lfloor \frac{b^2}{4}\rfloor$ by $\frac{b^2}{4}$ (the latter only when $b\geq 3$).
We observe that $N_2(a,b,c,d,r) > 0$ for all but finitely many exceptions listed in the lemma below.
If $b$ and $d$ have the same parity then $n$ is even and therefore we only need to show $N_2(a,b,c,d,r) > -\frac{r+1}{4}$.

\begin{lemma}\label{A1lem}
Assume $r\geq 7$, and that $a$, $b$ are not both zero.

a) $N_2(a,b,c,d,r) > 0$ for $(a,b,c,d)$ not belonging to the following list:
$$(0,1,0,0),\; (0,1,0,1),\;(0,1,0,2),\;(0,1,1,1),\;(0,1,2,0),\;(0,1,2,1),\;(0,1,2,2),\;(0,1,3,1),\;(0,2,0,2),$$
$$(0,2,1,2),\;(0,2,2,1),\;(0,2,2,2),\;(0,2,2,3),\;(0,3,2,3), \;(1,0,0,0),\;(1,1,0,0).$$

%b) $N(a,b,c,d,r)\geq -\frac{r+1}{4}$ for $(a,b,c,d)=(1,1,0,2)$.

b) $N_2(a,b,c,d,r) > -\frac{r+1}{4}$ for $(a,b,c,d)$ from the following list:
$$(0,1,1,1),\;(0,1,3,1),\;(0,2,0,2),\;(0,2,1,2),\;(0,3,2,3),\;(1,0,0,0).$$
\end{lemma}

\begin{proof}
We observe that $N_1$ and $N_2$ are linear in $r$ and, on each specified range of values of $a,b,c,d$, the coefficients of %$r$
{$N_i(r)$} (for brevity, we just write $N_i(r)$ instead of $N_i(a,b,c,d,r)$) are polynomials in $a,b,c,d$ of degree 2. So we set $$N_i(r)=A_i(a,b,c,d)r+B_i(a,b,c,d).$$
We want to show that $N_2(r) > 0$ for all $r\geq 7$.
We first verify that the coefficient $A_1(a,b,c,d)$ of $r$ is non-negative with one exception: when $(a,b,c,d)=(0,1,0,1)$.
To see this we computed $A_1(a,b,c,d)-\frac{(d-b)^2}{4}$, which equals a constant term (possibly negative) plus a sum of terms of degree at most 2 with positive coefficients.
One observes that $A_1(a,b,c,d)\geq \frac{(d-b)^2}{4}$ unless $(a,b,c,d)=(0,1,0,1)$, from which our assertion follows. Consequently, $A_2(a,b,c,d)\ge A_1(a,b,c,d)\geq 0$ unless $(a,b,c,d)=(0,1,0,1)$.

It follows that with the exception of the case $(a,b,c,d)=(0,1,0,1)$, we have $N_1(r)\geq N_1(7)$ and $N_2(r)\geq N_2(7)$ for all $r\geq 7$, and hence (for (a)) it suffices to show that $N_2(7) > 0$. Since $N_{2}(r)\ge N_{1}(r)$ for each $r$, it will also suffice to show $N_1(7)>0$.
Computing $N_1(7)$, we once more (locally) have a degree 2 polynomial in $a,b,c,d$ which we 
can prove is a sum of non-negative terms and a constant term. It turns out that $N_1(7)$ is positive for all but finitely many exceptions, for which we directly verify the inequality $N_2(7)>0$.
We will go through the two `generic' cases (i) $b\geq 3$, $c\geq \max \{2,a+b\}$, 
(ii) $b\geq 3$, $2\leq c\leq a+b$.
%The computation of $N(8)$ in each case, including the non-negative lower bound 
%expressed in terms of $a,b,c,d$, is given in Table \ref{typeAtable}.

\vspace{0.2cm}
\noindent
{\bf Case (i)}: 
Here $\dim C_{r-1}(\mathfrak{gl}_b)\leq \frac{rb^2}{3}+\frac{2r}{3}-1$ and $\dim {\mathcal 
Z}\leq \frac{a^2}{11}+c^2+2(a+b)c+(r-2)(c^2+\frac{(a+b)^2}{2})$.
Then the coefficient $A_1(a,b,c,d)$ of $r$ in $N_1(r)$ equals:
$$\frac{1}{2}a^2+2ac+\frac{5}{12}b^2+bc-\frac{1}{2}bd+\frac{1}{4}d^2-\frac{11}{12}=\frac{1}{2}a^2+(2a+b)c+\frac{1}{6}b^2+\frac{1}{4}(d-b)^2-\frac{11}{12}$$
which is positive since $\frac{1}{6}b^2\geq \frac{3}{2}$.
Thus $N_1(r)\geq N_1(7)$ for all $r\geq 7$.
Now we proceed to compute $N_1(7)$, which equals:
$$(6a+2b)(c-a-b)+2\left(d-b-\frac{a}{2}\right)^2+\frac{54}{11}a^2+ab+\frac{11}{12}b
^2-\frac{17}{3}.$$
Since $(6a+2b)(c-a-b)$ and $(d-b-\frac{a}{2})^2$ are non-negative, and since $\frac{11}{12}b^2\geq \frac{33}{4}>\frac{17}{3}$, it follows that $N_1(7) > 0$ and hence $N_1(r) > 0$ for all $r\geq 7$.

\vspace{0.2cm}
\noindent
{\bf Case (ii)}: 
Here $\dim C_{r-1}(\mathfrak{gl}_b)\leq \frac{rb^2}{3}+\frac{2r}{3}-1$ and $\dim {\mathcal 
Z}\leq \frac{a^2}{11}+c^2+2(a+b)c+(r-2)(\frac{c^2}{2}+(a+b)c)$.
The coefficient $A_1(a,b,c,d)$ of $r$ equals $\frac{1}{4}(d-b)^2+a^2+ab+ac+\frac{2}{3}b^2+\frac{1}{2}c^2-\frac{11}{12}$, which is positive once more since $\frac{2}{3}b^2\geq 6$.
We now proceed to compute $N_1(7)$, which equals:
$$2\left(d-b-\frac{a}{2}\right)^2 +\frac{17}{12}\left(b-\frac{18c}{17}-\frac{12a}{17}\right)^2+\frac{263}{374}\left(a-\frac{209c}{263}\right)^2+\frac{123}{263}c^2-\frac{17}{3}$$
which is positive for $c\ge 4$. On the other hand,
$$N_1(7)=2\left(d-b-\frac{a}{2}\right)^2 +\frac{17}{12}\left(b-\frac{18c}{17}-\frac{12a}{17}\right)^2+\frac{31}{34}\left(c-\frac{19a}{31}\right)^2+\frac{123}{341}a^2-\frac{17}{3}$$
which is positive for $a\ge 4$, and
$$N_1(7)=2\left(d-b-\frac{a}{2}\right)^2 +\frac{5}{2}\left(c-\frac{3b}{5}+\frac{a}{5}\right)^2+\frac{72}{55}\left(a-\frac{77b}{144}\right)^2+\frac{41}{288}b^2-\frac{17}{3}$$
which is positive for $b\ge 7$. Thus $N_1(7) > 0$ if $a\ge 4$, $b\ge 7$ or $c\ge 4$. It is also clear that $N_1(7)>0$ if $d-b-\frac{a}{2}\ge 2$, so $N_1(7)$ is positive for each $d\ge 10$. For each of the cases satisfying $a\le 3$, $3\le b\le 6$, $c\in \{2,3\}$ and $d\le 9$ we compute $N_2(7)$ and observe that it is positive unless $(a,b,c,d)=(0,3,2,3)$.

\vspace{0.1cm}
We deal similarly with the remaining cases, by inspecting:
$$N_1'(7)=\left\{ \begin{array}{cc}N_1(7)-2(d-b-\frac{a}{2})^2-\frac{3}{2}(a-\frac{2b}{3})^2 & \mbox{if $b\ge 3$ and $c\le 1$}\\
N_1(7)-2(d-b-\frac{a}{2})^2- (6a+2b)(c-a-b) & \mbox{if $b\le 2$ and $c\geq \max\{ 2,a+b\}$,} \\
N_1(7)-2(d-b-\frac{a}{2})^2 - \frac{11}{4}(b-\frac{4a+6c+10}{11})^2-\frac{23}{22}(a-\frac{c+20}{23})^2 & \mbox{if $b\le 2$ and $2\leq c<a+b$,} \\
N_1(7)-2(d-b-\frac{a}{2})^2 - \frac{3}{2}(a-\frac{2}{3}b+\frac{2}{3}c)^2 & \mbox{if $b\le 2$ and $c\leq 1$.}\end{array}\right.$$
We observe that $N_1'(7)>0$ except in the following cases: 

i) $a=0$, $b$ equals 1 or 2 and $c\geq 2$. Here $N_2(7)> 0$ unless $(a,b,c,d)$ belongs to the set $\{(0,1,2,0),(0,1,2,1),(0,1,2,2),(0,1,3,1),(0,2,2,1),(0,2,2,2),(0,2,2,3)\}$.

ii) $c=2$, $b\le 2$ and $a+b>2$, when $N_1'(7)=-\frac{91}{23}{+\frac{b^2}{4}-\lfloor\frac{b^2}{4}\rfloor}$. Here $N_1(7)>0$ unless $(a,b,c,d)\in \{(1,2,2,2),(1,2,2,3),(2,2,2,3)\}$. However, $N_2(7)>0$ also in these cases.

iii) $(b,c)=(2,1)$, when $N_1'(7)=-\frac{2}{3}$. Here $N_2(7)=\max\{2a^2-2ad+2d^2+2a-8d+8,2a^2+3ad+2d^2-\lfloor\frac{d^2}{2}\rfloor +2a-8d+8\}$, which is positive unless $(a,d)=(0,2)$.

iv) $(b,c)=(1,1)$, when $N_1'(7)=-1$.
Here $N_2(7)=\max\{2a^2-2ad+2d^2+2a-4d+1,2a^2+3ad+2d^2-\lfloor \frac{d^2}{2}\rfloor+2a-4d+1\}$, which is positive unless $(a,d)=(0,1)$.

v) $(a,b,c)=(1,0,1)$ when $N_1(7)=2d^2-2d+6>0$. 

vi) $(b,c)=(2,0)$, when $N_1'(7)=-\frac{11}{3}$. Here $N_2(7)=\max\{2a^2-2ad+2d^2-8d+7,2a^2+3ad+2d^2-\lfloor \frac{d^2}{2}\rfloor-8d+7\}$, which is positive unless $(a,d)=(0,2)$.

vii) $(b,c)=(1,0)$, when $N_1'(7)=-\frac{14}{3}$.
Here $N_2(7)=\max\{2a^2-2ad+2d^2-4d-2,2a^2+3ad+2d^2-\lfloor \frac{d^2}{2}\rfloor -4d-2\}$, which is positive unless $a=0$ and $d\leq 2$ or $a=1$ and $d=0$.

viii) $(b,c)=(0,0)$, when $N_1'(7)=-2$.
In this case $N_2(7)=\max\{2a^2-2ad+2d^2-2,2a^2+3ad+2d^2-\lfloor \frac{d^2}{2}\rfloor -2\}$, which is positive unless $a=1$ and $d=0$.

For the special cases in (b), we compute $N_2(r)+\frac{r+1}{4}$ and see that $N_2(7)>-2$ in that cases. Since the coefficient $A_2(a,b,c,d)$ is nonnegative, we get $N_2(r)+\frac{r+1}{4}>0$ for $r\geq 7$.
\end{proof}

To finish, we deal with the remaining special cases.

\begin{lemma}
For $r\ge 7$ the following holds:
\begin{enumerate}
\item
If $x\in \mathcal{N}(\mathfrak{g})$ has associated partition $[3,2^2,1^d]$, then $\dim C'(x)\le (r+1)(4d+10)+14<(r+1)\lfloor\frac{n^2}{4}\rfloor.$
\item
If $x\in \mathcal{N}(\mathfrak{g})$ has associated partition $[3^2,2^2,1^d]$ then $\dim C'(x)\le (r+1)(6d+20)+30<(r+1)\lfloor\frac{n^2}{4}\rfloor.$
\item
If $x\in \mathcal{N}(\mathfrak{g})$ has associated partition $[3,1^d]$ with $d\le 2$, then $\dim C'(x)\le (r+1)(d+2)+\lfloor \frac{d^2}{2}\rfloor +3d+2$. In particular, $\dim C'(x)<(r+1)\lfloor \frac{n^2}{4}\rfloor$ if $d>0$ and $\dim C'(x)\le (r+1)\frac{n^2}{4}$ if $d=0$.
\item
If $x\in \mathcal{N}(\mathfrak{gl}_7)$ has associated partition $[4,3]$, then $\dim C'(x)\le 8r+30<(r+1)\lfloor \frac{49}{4}\rfloor.$
\end{enumerate}
\end{lemma}

\begin{proof}
For (1) and (2) we have $a=0$, therefore the variety ${\mathcal{Y}_{r-1,a,b}}$ identified in the proof of Lemma \ref{GLblocksize4} is 0-dimensional. Moreover, $b\le 2$, therefore
$$\dim C_{r-1}'(\mathfrak{z}_{\mathfrak{g}}(x))\le (r-1)(b^2+4b+2bd+2d)+b(r+b-2)+\dim \mathcal{Z}_{r-1,0,b,2}.$$
By the proof of Lemma \ref{y,z,w,v,u} we get
$$\dim \mathcal{Z}_{r-1,0,b,2}\le \max_{m,l}\left(4+4b-2m^2-2ml+\lfloor \frac{l^2}{4}\rfloor -\lceil \frac{l^2}{2}\rceil +(r-2)(4-2m^2-2ml-l^2+b(2m+l))\right).$$
%(Note that in Lemma \ref{y,z,w,v,u} we have to replace $r$ by $r-1$ to get the variety $\mathcal{Z}$ identified in Lemma \ref{GLblocksize4}).
Since $2m+l\le 2$, we get $(m,l)\in \{(0,0),(0,1),(0,2),(1,0)\}$. It can be easily verified that the above maximum is equal $4r$ for $b=1$ and $6r-2$ for $b=2$. Since $n^2-\dim \mathfrak{z}_{\mathfrak{g}}(x)=6b^2+4bd+16b+4d+8$, the required estimates for $\dim C'(x)$ in (1) and (2) immediately follow. To show that $\dim C'(x){<} (r+1)\lfloor \frac{n^2}{4}\rfloor$ observe that $(r+1)\lfloor \frac{(d+7)^2}{4}\rfloor -(r+1)(4d+10)=(r+1)(\lfloor \frac{(d-1)^2}{4}\rfloor +2)\ge 2(r+1)>14$ and $(r+1)\lfloor \frac{(d+10)^2}{4}\rfloor -(r+1)(6d+20)=(r+1)(\lfloor \frac{(d-2)^2}{4}\rfloor +4)\ge 4(r+1)>30$ for $r\ge 7$.

For (3), we may assume
$$x=\begin{bmatrix} 0 & 1 & 0 & 0\\ 0 & 0 & 1 & 0 \\ 0 & 0 & 0 & 0 \\ 0 & 0 & 0 & 0
\end{bmatrix},\quad \mathrm{so}\quad \mathfrak{z}_{\mathfrak{g}}'(x)=\left\{\begin{bmatrix} 0 & u & v & z^T\\ 0 & 0 & u & 0 \\ 0 & 0 & 0 & 0\\ 0 & 0 & w & 0
\end{bmatrix};u,v\in k,z,w\in k^d\right\}.$$
It is clear now that $C_{r-1}'(\mathfrak{z}_{\mathfrak{g}}(x))$ is the product of the affine space $k^{2(r-1)}$ and the variety $\mathcal{W}_{r-1,1,d}$ defined in Lemma \ref{y_iz_j=y_jz_i}. Hence $\dim C_{r-1}'(\mathfrak{z}_{\mathfrak{g}}(x))\le 2(r-1)+rd+\lfloor \frac{d^2}{2}\rfloor$ and $\dim C'(x)\le (r+1)(d+2)+\lfloor \frac{d^2}{2}\rfloor +3d+2$.

For (4), we may assume
$$x=\begin{bmatrix} 0 & 1 & 0 & 0 & 0 & 0 & 0 \\0 & 0 & 1 & 0 & 0 & 0 & 0 \\0 & 0 & 0 & 1 & 0 & 0 & 0 \\0 & 0 & 0 & 0 & 0 & 0 & 0 \\0 & 0 & 0 & 0 & 0 & 1 & 0 \\0 & 0 & 0 & 0 & 0 & 0 & 1 \\0 & 0 & 0 & 0 & 0 & 0 & 0
\end{bmatrix}, \quad \mathrm{so} \quad \mathfrak{z}_{\mathfrak{gl}_7}(x)\cap \mathcal{N}(\mathfrak{gl}_7)=\left\{
\begin{bmatrix} 0 & a_1 & a_2 & a_3 & b_1 & b_2 & b_3 \\0 & 0 & a_1 & a_2 & 0 & b_1 & b_2 \\0 & 0 & 0 & a_1 & 0 & 0 & b_1 \\0 & 0 & 0 & 0 & 0 & 0 & 0 \\0 & c_1 & c_2 & c_3 & 0 & d_1 & d_2 \\0 & 0 & c_1 & c_2 & 0 & 0 & d_1 \\0 & 0 & 0 & c_1 & 0 & 0 & 0
\end{bmatrix};a_1,\ldots ,d_2\in k\right\}.$$
It is clear that a matrix as above belongs to $\mathfrak{z}_{\mathfrak{gl}_7}'(x)$ if and only if $b_1c_1=0$. Let $\pi \colon C_{r-1}'({\mathfrak{z}_\g(}x)\to k^{8(r-1)}$ be the projection sending $(y_1,\ldots ,y_{r-1})\in C_{r-1}'({\mathfrak{z}_\g(}x)$ to the collection of all $a_1$'s, $a_2$'s, $a_3$'s, $b_1$'s, $b_3$'s, $c_1$'s, $c_3$'s and $d_2$'s. The theorem on dimensions of fibres yields
$\dim C_{r-1}'({\mathfrak{z}_\g(}x)\le 7(r-1)+\dim \pi ^{-1}(0,\ldots ,0)$. However, $\pi ^{-1}(0,\ldots ,0)$ is isomorphic to the determinantal variety of all $3\times (r-1)$ matrices of rank {at most} 1 and is therefore of dimension $r+1$. Hence $\dim C_{r-1}'({\mathfrak{z}_\g(}x)\le 8r-6$ and $\dim C'(x)\le 8r+30$. We note that in the case $(a,b,c,d)=(1,1,0,0)$ we obtain $N_{2}(7)=0$. Since we are interested in showing that the set(s) defined in Cor. \ref{square_zeroC'(x)} is (are) the only irreducible component(s) of maximal dimension, this is not quite enough for our considerations when $r=7$.
\end{proof}

\end{appendices}

\bibliographystyle{alpha}
\bibliography{commvarsbib}

\end{document}